\documentclass[final,3p,times]{elsarticle}

\usepackage{amssymb}
\usepackage{graphicx,setspace}
\usepackage{psfrag}
\usepackage{amsfonts}
\expandafter\let\csname equation*\endcsname\relax
\expandafter\let\csname endequation*\endcsname\relax
\usepackage{amsmath}
\usepackage{amssymb, amsthm}
\usepackage{mathrsfs}
\usepackage{epstopdf}
\usepackage{float}
\usepackage{lineno,hyperref}
\usepackage{color}
\usepackage{subfigure}

\usepackage{amsmath}
\usepackage{dsfont} 
\usepackage{amssymb} 
\usepackage{graphicx,color}
\usepackage{extarrows}

\usepackage{graphicx}

\usepackage{epstopdf}

\journal{Stochastics and Dynamics}

\begin{document}
\newtheorem{definition}{Definition}[section]
\newtheorem{lemma}{Lemma}[section]
\newtheorem{remark}{Remark}[section]
\newtheorem{theorem}{Theorem}[section]
\newtheorem{proposition}{Proposition}
\newtheorem{assumption}{Assumption}
\newtheorem{example}{Example}
\newtheorem{corollary}{Corollary}[section]
\def\ep{\varepsilon}
\def\Rn{\mathbb{R}^{n}}
\def\Rm{\mathbb{R}^{m}}
\def\E{\mathbb{E}}
\def\hte{\hat\theta}
\renewcommand{\theequation}{\thesection.\arabic{equation}}
\begin{frontmatter}



\title{Stochastic turbulence for Burgers equation driven by
cylindrical L\'evy process}

\author{Shenglan Yuan\fnref{addr1,addr2}}\ead{shenglan.yuan@math.uni-augsburg.de}
\author{Dirk Bl$\rm\ddot{o}$mker\fnref{addr1}}\ead{dirk.bloemker@math.uni-augsburg.de}
\author{Jinqiao Duan\fnref{addr2,addr3}}\ead{duan@iit.edu}

\address[addr1]{\rm Institut f$\rm\ddot{u}$r Mathematik, Universit$\rm\ddot{a}$t Augsburg,
86135, Augsburg, Germany }
\address[addr2]{\rm Center for Mathematical Sciences, Huazhong
University of Science and Technology, 430074, Wuhan, China}
\address[addr3]{\rm Department of Applied Mathematics, Illinois Institute of Technology, Chicago, Illinois 60616, USA}

\begin{center}
{\small \textbf{Dedicated to Bj\"{o}rn Schmalfuss on the occasion of his 65th birthday}}
\end{center}

\begin{abstract}
This work is devoted  to investigating stochastic turbulence for the fluid flow in one-dimensional viscous Burgers equation perturbed by L\'evy space-time white noise with the periodic boundary condition. We rigorously discuss the regularity of solutions and their statistical quantities in this stochastic dynamical system.  The quantities  are such as the moment estimate, the structure function and the energy spectrum of the turbulent velocity field. Furthermore, we provide qualitative and quantitative properties of  the stochastic Burgers equation when the kinematic viscosity $\nu$ tends towards zero. The inviscid limit describes the strong stochastic turbulence.
\end{abstract}


\end{frontmatter}


\section{Introduction}
Burgers equation is a dissipative system occurring in various areas of applied mathematics, such as fluid mechanics \cite{BRS,G17}, nonlinear acoustics \cite{GBPC}, gas dynamics \cite{M10}, and traffic flow \cite{W05}. For a given kinematic viscosity $\nu$,  the viscous Burgers equation is
\begin{equation*}
\partial_tu=\nu\partial_{xx}u-u\partial_xu,\quad u(0,x)=u_{0}(x),
\end{equation*}
which describes the speed of the fluid at each location along the pipe as time progresses \cite{GS}. The solution is then called Burgers turbulent fluid flow, and referred to as ``Burgers turbulence" or ``Burgulence'' \cite{FB}. We refer to it for the case of a fluid flow $u(t,x)$ of order one, space-periodic of
period one.

By means of the celebrated Cole-Hopf transformation $u=-2\nu\frac{1}{\phi}\partial_x\phi$,
we convert the viscous Burgers equation into a linear equation
\begin{equation*}
\partial_x\Big(\frac{1}{\phi}\partial_t\phi\Big)=\nu\partial_x(\frac{1}{\phi}\partial_{xx}\phi),
\end{equation*}
which can be integrated with respect to $x$ on a periodic domain to get the diffusion equation $\partial_t\phi=\nu\partial_{xx}\phi$.
The heat equation can be solved explicitly by the convolution of the initial data and the heat kernel, and then we use the inverted Cole-Hopf transformation to obtain the solution to the viscous Burgers equation:
\begin{equation*}
 u(t,x)=-2\nu\partial_x\left(\ln \left\{(4\pi \nu t)^{-1/2}\int _{-\infty }^{\infty }\exp \left[-{\frac {(x-\hat{x})^{2}}{4\nu t}}-{\frac {1}{2\nu }}\int _{0}^{\hat{x}}f(\tilde{x})d\tilde{x}\right]d\hat{x}\right\}\right).
\end{equation*}

The Reynolds number of the flow $u(t,x)$ is $\text{\emph{\textbf{Rey}}}= \nu^{-1}$,
where $\nu$ is the kinematic viscosity of fluid. If \emph{\textbf{Rey}} is large (i.e., $\nu\ll1$), then the velocity field $u(t,x)$  of the flow becomes
very irregular, i.e., turbulent \cite{LP}. 

In particular, when $\nu=0$, the viscous Burgers equation becomes the inviscid Burgers equation $\partial_tu=-u\partial_xu,$
which is a conservation law, more generally a first order quasilinear hyperbolic equation. It is important to stress that $u(t,x)$
converges as $\nu\rightarrow0$ to a weak solution for the inviscid Burgers equation constructed by the method of characteristics; in general the weak solution is not continuous  and its discontinuties are indeed very useful to express the turbulence due to an anomaly from the shock wave.

Taking the Burgers equation as a model for one-dimensional turbulence we follow Burgers\cite{Burg}.
In order to characterize  Burgers turbulence in the presence of random forces, we add  L\'evy noise $\eta$ to generate a stochastic Burgers equation. The main focus of this study is to investigate the qualitative properties for stochastic viscous Burgers equation driven by
cylindrical L\'evy process with the periodic boundary condition, bearing in mind,
\begin{equation}\label{B}
\partial_tu=\nu\partial_{xx}u-u\partial_xu+\eta(t,x),\quad u(0,x)=u_{0}(x), t\geq0, x\in \mathbb{S}^{1}, 0<\nu\leq1.
\end{equation}

The Kolmogorov theory of turbulence was created by A. N. Kolmogorov in three articles \cite{K41}-\cite{K41b} published
in 1941 (Kolmogorov's 1941 theory is referred to herein as K41 theory). It
describes statistical properties of turbulent flows and is now the most popular theory of
turbulence, known as the local similarity or universal equilibrium theory of small-scale velocity fluctuations in high Reynolds
number, incompressible and stationary turbulence. Kolmogorov proved that short-scale-in-$x$ features of a turbulent flow $u(t,x)$ display a universal
behaviour which depends on particularities of the system only through a few parameters. The Kolmogorov theory is statistical. That is, it assumes that the velocity field $u(t,x)$ depends on a random parameter $\omega\in(\Omega,\mathcal{F},\mathbb{P})$.  Moreover, Kolmogorov supposed that the random
field $u^{\omega}(t,x)$ is stationary in time and homogeneous.  It means that statistically $u^{\omega}(t,x)$ and $u^{\omega}(t+C,x+\tilde{C})$, where $C$ and $\tilde{C}$ are constants, are the same random field. Kolmogorov studied short space-increments
\begin{equation*}
u(t,x+l)-u(t,x),\quad|l|\ll1,
\end{equation*}
and examined moments of these random variables as functions of $|l|$. He also took
the Fourier coefficients $\hat{u}_k(t)$ of $u(t,x)$ and identified their second moments as functions of $|k|$. The Kolmogorov theory admits a natural one-dimensional version \cite{K62}.

In the spirit of K41 theory, Kuksin \cite{K97}-\cite{K67} examined the  properties of the turbulence limit for some nonlinear partial differential equations perturbed by a bounded random kick-force as the viscosity goes to zero. Boritchev \cite{AB14} obtained sharp estimates for the dissipation length scale and the
small-scale quantities which characterise the decaying Burgers turbulence, i.e., the
structure functions and the energy spectrum.  The references \cite{BK,DH19} focused on the turbulence in stochastic viscous Burgers equation with forcing of Gaussian white noise type. The paper \cite{DK} discussed a number of rigorous results in the stochastic model for wave turbulence due to Zakharov-L'vov, and considered the damped/driven (modified) cubic nonlinear Schr\"{o}dinger equation on a large torus and decomposed its solutions to formal series in the amplitude.
E, Khanin, Masel and Sinai \cite{E20} established existence and uniqueness of an invariant measure for the
Markov process corresponding to the inviscid stochastic Burgers equation, and also gave a detailed description
of the structure and regularity properties for the solutions that live on the support of this measure.

Experimentally, Gottwald \cite{G07} evaluated centred second order in time and space discretizations of the inviscid
Burgers equation that supports non-smooth shock wave solutions in its continuum formulation, and derived the
modified equation associated with the numerical scheme using backward error analysis. LaBryer, Attar and Vedula \cite{LAV} improved
the reliability and computational efficiency of large-eddy simulation predictions for turbulent flows of the forced Burgers equation, and  constructed subgrid models based upon information that is consistent with the underlying spatiotemporal statistics of the flow. Bl\"{o}mker, Kamrani and Hosseini \cite{BKH} investigated the spectral Galerkin method for spatial discretization and the rate
of convergence in uniform topology for the stochastic Burgers equation driven by colored noise with well-illustrated numerical examples.
A related work for the numerical solution of the one-dimensional Burgers equation with Neumann boundary noise using the Galerkin
approximation in space and the exponential Euler method in time was referred to Ghayebi et al., \cite{GHB}.

Most of our results rely on some knowledge for cylindrical
L\'evy process, e.g. It\^{o}'s formula. The notion cylindrical L\'evy process appeared the first
time in Peszat and Zabczyk \cite{PZ}. The systematic
introduction of cylindrical L\'evy processes was presented Applebaum
and Riedle \cite{AR10}. The independent and stationary increments of a cylindrical L\'evy process enable us
to identify the cylindrical L\'evy process as a good integrator. Bertoin \cite{B01} established some properties and elements of Burgers turbulence with white noise initial data, and discussed their possible generalization to the stable noise case. Truman and Wu \cite{TW03} proved existence of a unique, local and mild solution to the stochastic Burgers equation driven by L\'evy space-time white noise. Dong, Xu and Zhang \cite{DXZ} identified the strong Feller property and the exponential ergodicity of stochastic Burgers equations driven by $\alpha$-stable processes.

This work is structured as follows. In Section \ref{AA}, we introduce analysis tools including Sobolev space, mass conserving noise and L\'evy
processes with bounded jumps that have moments of all orders. We regard the solution $u(t,\omega,x)$ of \eqref{B} as a turbulent curve depending on $\omega$ in $L^{2}(\mathbb{S}^{1})$, i.e., as a random processes in Sobolev spaces. In Section \ref{Well}, we present the properties of existence
and uniqueness for the pathwise solution $u$ of model \eqref{B}.  In Section \ref{OK}, we discuss the Oleinik-Kruzkov inequality about estimates on $u$ in $L^{\infty}$, and on $\partial_{x}u$ in $L^{1}$. In Section \ref{ME}, we systematically average various functionals $f$ of solution $u$ to estimate quantities
\begin{equation*}
\mathbb{E}f(u(t))=\int_{\Omega}f(u(t,\omega))\mathbb{P}(\omega),\quad f: H^{n}\rightarrow\mathbb{R},
\end{equation*}
by means of the logic in the turbulence theory. Then, we get upper and lower bounds for the
Sobolev norms of solution in stochastic Burgers equation \eqref{B}.
In Section \ref{ST}, we deduce the main results consisting of
the structure function in $x$ and the energy spectrum in Fourier for the fluid flow $u$, and justify the non-Gaussian behavior of stochastic turbulence by considering small-scale quantities. In Section \ref{SI}, we show that the energy spectrum of solution $u$ in dynamical system \eqref{B}
is of the form of Kolmogorov's law $E_{n}(u)\sim n^{-2}$, $n\in\mathbb{N}^{\ast}$ when $\nu\rightarrow0$ based on the results of the previous sections. Finally, in Section \ref{five}, we summarize our conclusions and challenges, as well as a number
of directions for future study.

\section{Analysis tools}\label{AA}
Throughout the paper, we shall work in the separable real Hilbert space
\begin{equation*}
H=\{\phi\in L^{2}(\mathbb{S}^{1}):\int_{\mathbb{S}^{1}}\phi(x)dx=0\},
\end{equation*}
and equip it with scalar product $\langle\cdot,\cdot\rangle$ and norm $\|\cdot\|$ defined by
\begin{equation*}
\langle\phi,\tilde{\phi}\rangle:=\int_{\mathbb{S}^{1}}\phi(x)\tilde{\phi}(x)dx,\quad\|\phi\|:=\Big(\int_{\mathbb{S}^{1}}\phi^{2}(x)dx\Big)^{\frac{1}{2}},\quad\phi,\tilde{\phi}\in H.
\end{equation*}
Note that $\mathbb{S}^{1}$ is identified with the interval $[0,2\pi)$ and periodic boundary condition. The standard scalar product
in $L^2(\mathbb{S}^{1})$ is denoted by $\langle\cdot,\cdot\rangle_{L^2}$. For the Laplace operator $\Delta$ on $H$ with $D(\Delta)=W^{k,2}(\mathbb{S}^{1})\cap H$. The eigenfunctions of $\Delta$ are
\begin{eqnarray}\label{try}
e_{k}(x)=\left\{\begin{array}{l}
\sqrt{2}\cos(2\pi k x),\quad\text{if}~~k\in\{1,2,\cdot\cdot\cdot\};\\
\sqrt{2}\sin(2\pi k x),\quad\text{if}~~k\in\{-1,-2,\cdot\cdot\cdot\}.
\end{array}
\right.
\end{eqnarray}
Those functions $\{e_{k},k=\pm1,\pm2,\cdot\cdot\cdot\}$ form the trigonometric basis  in  $H$ of periodic function
with zero mean. Note that if $u(t,x)=\sum_{k\in\mathbb{Z}^{\ast}}u_{k}(t)e_{k}(x)$, then by $e^{2i\pi kx}=\cos(2\pi k x)+i\sin(2\pi k x)$, we write it as Fourier series
\begin{equation*}
u(t,x)=\sum_{k\in\mathbb{Z}^{\ast}}\hat{u}_{k}(t)e^{2i\pi kx},
\end{equation*}
where $\hat{u}_{k}(t)=\bar{\hat{u}}_{-k}(t)=(\sqrt{2})^{-1}(u_{k}(t)-iu_{-k}(t))$ is called Fourier coefficient.

For all $k\in\mathbb{Z}^{\ast}:=\mathbb{Z}\backslash\{0\}$,
\begin{equation*}
\Delta e_{k}=-\lambda_{k}e_{k}\quad\text{with}\quad\lambda_{k}=4\pi^{2}|k|^{2}.
\end{equation*}
For any $\theta\in \mathbb{R}$, by using the domain of definition for fractional powers of the operator $\Delta$, we define
\begin{equation*}
H^{\theta}:=D((-\Delta)^{\theta/2}):=\Big\{\phi=\sum_{k\in\mathbb{Z}^{\ast}}\phi_{k}e_{k}: \phi_k=\langle \phi,e_k\rangle\in \mathbb{R}, \sum_{k\in\mathbb{Z}^{\ast}}\lambda_{k}^{\theta}\phi_{k}^{2}<\infty\Big\},
\end{equation*}
and
\begin{equation*}
(-\Delta)^{\theta/2}\phi:=\sum_{k\in\mathbb{Z}^{\ast}}\lambda_{k}^{\theta/2}\phi_ke_k,~~~\phi\in D((-\Delta)^{\theta/2}),
\end{equation*}
with the associated norm
\begin{equation*}
\parallel \phi\parallel_{\theta}=\parallel \sum_{k=1}^{\infty}\phi_ke_k\parallel_{\theta}:=\parallel(-\Delta)^{\theta/2}\phi\parallel:=\sqrt{\sum_{k\in\mathbb{Z}^{\ast}}\lambda_{k}^{\theta}\phi_{k}^{2}} .
\end{equation*}
It is worthwhile to note that when $\theta=n\in\mathbb{N}^{*}$, $H^n$ is equivalent to the Sobolev space
$H^n=\{u\in H,\, u^{(n)}\in H\}$,
endowed with the homogeneous scalar product
\begin{equation*}
\langle u,v\rangle_{n}=\int_{\mathbb{S}^{1}}u^{(n)}(x)v^{(n)}(x)dx,
\end{equation*}
where $u^{(n)}:=\partial_{x}^{n}u$ represents the weak derivative w.r.t. $x$ of order $n$ for $u$. The norm $\|\cdot\|_{n}$ is induced by the product, and $\partial_{x}:H^{n+1}\rightarrow H^n$ is an isomorphism satisfying $\|\partial_{x}^{m}u\|_{n}=\|\partial_{x}u\|_{m+n}$, for all $m,n\in\mathbb{N}$.

If $u\in H$, then $u$ is written as $u(x)=\sum_{k\in\mathbb{Z}^{\ast}}u_{k}e_{k}(x)$ using the trigonometric base \eqref{try}, and the norm
of $u$ in $H^n$ is
\begin{equation*}
\|u\|_{n}^{2}=(2\pi)^{2n}\sum_{k\in\mathbb{Z}^{\ast}}|k|^{2n}|u_{k}|^{2}=2(2\pi)^{2n}\sum_{k=1}^{\infty}|k|^{2n}|\hat{u}_{k}|^{2}.
\end{equation*}
Alternatively, we use this characterization to determine the Sobolev space $H^n$ for all $n\geq0$:
$H^n=\{u\in H, \|u\|_{n}<\infty\}$. For $n<0$, we define $H^n$ as the completion of $H$ with respect to the norm $\|\cdot\|_{n}$.

We recall the Sobolev injection:
\begin{equation}\label{So}
H^n\hookrightarrow C^{k}(\mathbb{S}^{1})\Longleftrightarrow n>k+\frac{1}{2},
\end{equation}
and that the space $H^n$ for $n>\frac{1}{2}$ is a Hilbertian algebra:
\begin{equation}\label{uvn}
\|uv\|_{n}\leq c_{n}\|u\|_{n}\|v\|_{n}, \quad\quad\quad\text{for some constant}~~c_n.
\end{equation}

Let 
$p\in[1,+\infty)$ 
. For a function $u(x)$, we denote by $|u|_{p}$ its norm in the Lebesgue space $L^{p}(\mathbb{S}^{1})$, and by $\|u\|_{n}$ its homogeneous Sobolev norm of order $n$. If $\theta=0$, then we write $\|u\|:=\|u\|_{0}=|u|_{2}$.

In order to guarantee mass conservation $\int_{\mathbb{S}^{1}}u(t,x)dx=0$ of the solution to \eqref{B}, we suppose that the initial value and the noise are mass conserving \cite{BKW}, i.e., for all $t\geq0$, $\int_{\mathbb{S}^{1}}u_{0}dx=\int_{\mathbb{S}^{1}}\eta dx=0$.
The stochastic term $\eta=\partial_tL(t,x)$ is L\'evy space-time white noise \cite{NOP}, where $L$ is a cylindrical L\'evy process defined via
\begin{equation}\label{decom}
L(t)=\sum_{k=\pm1,\pm2,\cdot\cdot\cdot}\beta_{k}L_{k}(t)e_{k}(x).~~~~~~
\end{equation}
Here $\{\beta_{k}\}$ is a given sequence of real numbers, converging to zero sufficiently fast, i.e., there exist constants $C_{1}, C_{2}>0$ such that, for any $k\in\mathbb{Z}^{\ast}$,
\begin{equation*}
C_1\lambda_{k}^{-\gamma_0}\leq\beta_{k}\leq C_2\lambda_{k}^{-\gamma_0},~~~~\text{with}~~\gamma_0>1.
\end{equation*}
The independent one-dimensional L\'evy processes $\{L_{k}(t)\}_{k\in\mathbb{Z}^{\ast}}$ on a standard probability space $(\Omega,\mathcal{F},\{\mathcal{F}_t\}_{t\geq0},\mathbb{P})$, 
 have the same characteristic function by L\'evy-Khinchine formula,  satisfying that for any $k\in\mathbb{Z}^{\ast}$ and $t\geq0$,
\begin{equation*}
\mathbb{E}[e^{i\xi L_{k}(t)}]=[\mathbb{E}e^{i\xi\cdot L_{k}(1)}]^{t}=e^{t\psi(\xi)},\,\,\,\xi\in \mathbb{R},
\end{equation*}
where $\psi(\xi)$ is the L\'evy symbol given by
$$\psi(\xi)=ib\xi-\frac{1}{2}\sigma^{2}\xi^{2}+\int_{y\neq0}\big(e^{i\xi y}-1-i\xi y\textbf{1}_{(0,1)}(|y|)\big)\mu(dy),\quad b,\sigma \in\mathbb{R}.$$
Note that $\mu$ in the triplet $(b,\sigma^{2},\mu)$ is the L\'evy measure satisfying $\int_{{R}\setminus{\{0\}}}1\wedge|y|^{2}\mu(dy)<\infty$. For $t>0$ and $B\in\mathcal{B}({\mathbb{R}}\setminus{\{0\}})$, define the Poisson random measure of $L_{k}(t)$ by
 \begin{equation*}
 N_{k}(t,B)=\sum_{0\leq s\leq t}\textbf{1}_{B}(\Delta L_{k}(s))=\#\{0\leq s\leq t: \Delta L_{k}(s)\in B\},
 \end{equation*}
where $\Delta L_{k}(s):=L_{k}(s)-L_{k}(s-)$ are  the jumps of the process $L_{k}$, and $L_{k}(s-)$ is the left limit of $L_{k}(s)$ at time $s$. The function $\mu(B)=\mathbb{E}(N(1,B))$ describes the expected number of jumps in a certain size at a time interval $(0,1]$ in particular. Furthermore, define the  compensated Poisson measure of $L_{k}(t)$ via $\tilde{N}_k(t,B)=N_k(t,B)-t\mu(B)$.
According to the L\'evy-It\^o decomposition, $L_k(t)$ can be expressed as \cite{DH17}
\begin{equation}\label{Ito-Dec}
 L_k(t)=tb+\sigma W_{t}+\int_{|y|<1}y\tilde{N}_k(t,dy)+\int_{|y|\geq 1}yN_k(t,dy),
\end{equation}
where $W$ is a one-dimensional standard Brownian motion.
In this work we only consider L\'evy processes $\{L_{k}(t)\}_{k\in\mathbb{Z}^{\ast}}$ with uniformly bounded jumps and L\'evy triplet $(0,0,\mu)$. They admit moments of all orders.

\begin{theorem}
If $\{L_{k}(t)\}_{k\in\mathbb{Z}^{\ast}}$ are L\'evy processes such that $|\Delta L_k(t,\omega)|\leq c$ for all $t\geq0$ and some constant
$c>0$, then $\mathbb{E}(|L_k(t)|^{p})<\infty$ for all $p\geq0$.
\end{theorem}
\begin{proof}
Let $\mathcal{F}_{t}^{L_k}:=\sigma(L_{k}(s),s\leq t)$ be the $\sigma$-algebra generated by $L_k$, and define the stopping times
\begin{equation*}
\tau_{0}:=0,\quad\tau_{n}:=\inf\{t>\tau_{n-1}:|L_{k}(t)-L_{k}(\tau_{n-1})|\geq c\}.
\end{equation*}
Since $\{L_{k}(t)\}_{k\in\mathbb{Z}^{\ast}}$ have c\`{a}dl\`{a}g paths, $\tau_{0}<\tau_{1}<\tau_{2}<\cdot\cdot\cdot$. Moreover, L\'evy processes enjoy the strong Markov property
$\tau_{n}-\tau_{n-1}\sim\tau_{1}$ and $\tau_{n}-\tau_{n-1}$ are stochastically independent of $\mathcal{F}_{\tau_{n-1}}^{L_k}$, i.e., $(\tau_{n}-\tau_{n-1})_{n\in\mathbb{N}}$ is an independent and identically distributed sequence. Therefore,
\begin{equation*}
\mathbb{E}e^{-\tau_{n}}=(\mathbb{E}e^{-\tau_{1}})^{n}=q^{n},
\end{equation*}
for some $q\in[0,1)$.  From the definition of the stopping times we infer
\begin{equation*}
|L_k(t\wedge\tau_{n})|\leq\sum_{m=1}^{n}|L_{k}(\tau_{m})-L_{k}(\tau_{m-1})|\leq\sum_{m=1}^{n}\big(\underset{\leq c}{\underbrace{|\Delta L_{k}(\tau_{m})|}}+\underset{\leq c}{\underbrace{|L_{k}(\tau_{m}-)-L_{k}(\tau_{m-1})|}}\big)\leq2nc.
\end{equation*}
Thus, $|L_{k}(t)|>2nc\Longrightarrow\tau_{n}<t$, and by Markov's inequality
\begin{equation*}
\mathbb{P}(|L_{k}(t)|>2nc)\leq\mathbb{P}(\tau_{n}<t)\leq e^{t}\mathbb{E}e^{-\tau_{n}}=e^{t}q^{n}.
\end{equation*}
Finally,
\begin{align*}
\mathbb{E}(|L_{k}(t)|^{p})&=\sum_{n=0}^{\infty}\mathbb{E}(|L_{k}(t)|^{p}\textbf{1}_{\{2nc<|L_{k}(t)|\leq2(n+1)c\}})\\
                       &\leq(2c)^{p}\sum_{n=0}^{\infty}(n+1)^{p}\mathbb{P}(|L_{k}(t)|>2nc)\leq(2c)^{p}e^{t}\sum_{n=0}^{\infty}(n+1)^{p}q^{n}<\infty.
\end{align*}
\end{proof}
\begin{remark}
A L\'evy process $L_{k}$ starting at zero, has stationary and independent
increments and is continuous in probability. All paths $[0,\infty)\ni t\mapsto L_{k}(t,\omega)$ are right-continuous with finite left-hand limits (c\'adl\'ag). For each $\omega$ the noise $L(t,\omega,x)$ in \eqref{B} defines a c\'adl\'ag curve $L(t,\cdot)\in\mathcal{D}^{0}(\mathbb{R}^{+},\mathcal{H}^{\theta}),\theta\geq0.$
\end{remark}
\begin{example}
For example, $L_{k}(t)$ is a L\'evy process with L\'evy measure $\mu(dy)=\frac{\tau(y)}{|y|^{1+\alpha}}dy$, $\alpha\in(1,2)$,
where $\tau(y)$ is smooth truncation function $\tau(y)=1$ for $|y|\leq1$, $\tau(y)=0$ for $|y|\geq2$. We suppress big jumps,
so that $\mathbb{E}(|L_k(t)|^{p})<\infty$, $p\geq0$, but $L_{k}(t)$ is close to $\alpha$-stable L\'evy process.
\end{example}

In a formal way (via multiplying by $dt$) we are able to rewrite \eqref{B} as an abstract stochastic evolution equation
\begin{equation}\label{FB}
du=(\nu\partial_{xx}u-u\partial_xu)dt+dL_t,\quad u(0)=u_0.
\end{equation}
We need to give a rigorous meaning to the formal equation \eqref{FB}. It is the abbreviation of the corresponding integral equation,
\begin{equation}\label{FBI}
u(t)=u_0+\int_0^t(\nu\partial_{xx}u-u\partial_xu)ds+L(t),\quad \forall\,t\geq0,\quad \forall\,x\in\mathbb{S}^{1},
\end{equation}
which is defined in a distributional sense.

\begin{definition}
We say that $u^{\omega}(t,x)$ is a unique pathwise solution of \eqref{B} if the stochastic differential equation \eqref{FBI} holds for each $\omega\in Q$, where $Q$ is a negligible set (i.e., $Q\in\mathcal{F}$ such that $\mathbb{P}(Q)=0$).
If $u_{0}\in H^{\theta},\,\theta\geq0$, then there is a unique pathwise solution $u^{\omega}\in\mathcal{D}(\mathbb{R}^{+},H^{\theta})$ of \eqref{B}. It depends on the random parameter $\omega$.
\end{definition}
\begin{remark}
Define a bilinear operator $Q(u,v):=\frac{1}{2}uv'+\frac{1}{2}u'v, u,v\in H^{1}$ and write $Q(u)=Q(u,u)$ using symmetry. We rewrite \eqref{FB} into the following stochastic Burgers equation driven by $L(t)$ as an alternative representation:
\begin{equation*}
du=[\nu\Delta u-Q(u)]dt+dL(t),\quad u(0)=u_0\in H^{\theta},\,\theta\geq1.
\end{equation*}
The mild formulation provides the basic ingredients to verify the existence and uniqueness of solution by the variation of constants formula.
 We call a stochastic process $u\in\mathcal{D}(\mathbb{R}^{+},H^{\theta}),\theta\geq0$ a mild solution of \eqref{FB} with initial condition $u(0)=u_0$ if
\begin{equation*}
u(t)=u_0e^{t\nu\Delta}-\int_{0}^{t}e^{(t-s)\nu\Delta}Q(u(s))ds+\int_{0}^{t}e^{(t-s)\nu\Delta}dL(s)
\end{equation*}
for all $t\in[0,T]$ a.s. Based on the regularity of the stochastic convolution \cite{R14}
\begin{equation*}
\int_{0}^{t}e^{(t-s)\nu\Delta}dL_s=\sum_{k\in\mathbb{Z}^{\ast}}\beta_{k}\int_{0}^{t}e^{(t-s)\nu\lambda_{k}}dL_{k}(s)e_{k},
\end{equation*}
we are able to solve the stochastic Burgers equation using Banach's fixed point argument.
\end{remark}

\section{Well-posedness of stochastic Burgers equation}\label{Well}
For $T>0$ and $n\in\mathbb{N}$, we endow the Banach space $\mathcal{D}([0,T];H^{n})$ with the uniform norm $\|u\|_{\mathcal{D}_T^n}=\sup\limits_{t\in[0,T]}\|u(t)\|_{n}$ since every c\`{a}dl\`{a}g function on $[0,T]$ is bounded.

Consider the stochastic heat equation
\begin{equation}\label{CH}
\partial_tv=\nu\partial_{xx}v+\eta(t,x),\quad v(0,x)=u_{0}(x).
\end{equation}
The decomposition \eqref{decom} of the force $L(t)$ in the trigonometric base
 \eqref{try} demonstrates that there exists a unique pathwise solution $v\in\mathcal{D}([0,T];H^{n})$ of stochastic parabolic partial differential equation \eqref{CH}, i.e., for all $t\in[0,T]$,
\begin{equation*}
v(t)=u_{0}+\nu\int_{0}^{t}\partial_{xx}v(s)ds+L(t),\quad u_0\in H^{n},
\end{equation*}
by explicitly calculating the Fourier coefficients of $v$.

Now we decompose a solution of model \eqref{B} into
\begin{equation*}
u(t,x)=v(t,x)+w(t,x),
\end{equation*}
where $v$ is the solution of system \eqref{CH}. Thus $w$ satisfies the perturbed Burgers equation
\begin{equation}\label{DB}
\partial_tw=\nu\partial_{xx}w-\frac{1}{2}\partial_{x}(w+v)^{2},\quad w(0)=0.
\end{equation}
The advantage of reducing from equation \eqref{B} to \eqref{DB} is the fact that all the coefficients of the latter are regular. In the following we solve the equation \eqref{DB}. We start with a lemma to deal with functional inequality.

\begin{lemma}\label{XY}
Let $n\in\mathbb{N}^{\ast}$, $q\in[1,\infty]$ and $w\in H^{n+1}$. Then, there exists $C(n)>0$ such that
\begin{equation}\label{W}
|\langle\partial_{x}^{2n}w,\partial_{x}w^{2}\rangle|\leq C(n)\|w\|_{n+1}^{1+\epsilon}|w|_{q}^{2-\epsilon},\quad \epsilon(n,q)=\frac{n+\frac{2}{q}-\frac{1}{2}}{n+\frac{1}{q}+\frac{1}{2}}.
\end{equation}
\end{lemma}
\begin{proof}
Using Leibniz's formula, we have
\begin{equation}\label{Cm}
|\langle\partial_{x}^{2n}w,\partial_{x}w^{2}\rangle|\leq C(n)\sum_{m=0}^{n}\int_{\mathbb{S}^{1}}|w^{(m)}w^{(n-m)}w^{(n+1)}|dx.
\end{equation}
We maximize the integral at the right hand side of \eqref{Cm} by H\"{o}lder's inequality to obtain
\begin{equation*}
\int_{\mathbb{S}^{1}}|w^{(m)}w^{(n-m)}w^{(n+1)}|dx\leq|w^{(m)}|_{p_{1}}|w^{(n-m)}|_{p_{2}}\|w\|_{n+1},
\end{equation*}
where $\frac{1}{p_{1}}+\frac{1}{p_{2}}=\frac{1}{2}$. Utilizing Gagliardo-Niremberg inequality, we get
\begin{equation*}
\int_{\mathbb{S}^{1}}|w^{(m)}w^{(n-m)}w^{(n+1)}|dx\leq C\|w\|_{n+1}^{1+\epsilon}|w|_{q}^{2-\epsilon},
\end{equation*}
where $\epsilon=\epsilon(m,q,p_{1},n+1)+\epsilon(n-m,q,p_{2},n+1)=\frac{m+\frac{2}{q}-\frac{1}{2}}{m+\frac{1}{q}+\frac{1}{2}}$.
Combing the last inequality and relation \eqref{Cm}, we deduce \eqref{W}.
\end{proof}

\begin{theorem}\label{eu}
Let $n\geq1$, $T>0$ and $v\in\mathcal{D}([0,T];H^{n})$. Then there is a unique pathwise solution $w\in\mathcal{D}([0,T];H^{n})$
of equation \eqref{DB} and there exists $C_n(T,\nu,\|v\|_{\mathcal{D}_T^n})>0$ such that
\begin{equation}\label{wn}
\|w\|_{\mathcal{D}_T^n}^{2}+\int_{0}^{T}\|w(t)\|_{n+1}^{2}dt\leq C_n.
\end{equation}
 \end{theorem}

\begin{proof}
Based on Galerkin's approximation for the stochastic Burgers equation \cite{BJ}, we immediately build the existence of $w$, and thus the technical details are omitted. As for uniqueness, let $w_1, w_2\in\mathcal{D}([0,T];H^{n})$ be solutions of \eqref{DB}, i.e.,
\begin{equation*}
\partial_tw_i=\nu\partial_{xx}w_i-\frac{1}{2}\partial_{x}(w_i+v)^{2},\quad w_i(0)=0,~~i=1,2.
\end{equation*}
Then the difference, $w:=w_1-w_2$ satisfies the equation
\begin{equation*}
\partial_tw=\nu\partial_{xx}w-\frac{1}{2}\partial_{x}((w_1+w_2+2v)w),\quad w(0)=0.
\end{equation*}
Taking the scalar product in $\langle\cdot,\cdot\rangle_{L^2}$ of this previous equation with respect to $w$, and doing integration by parts,
\begin{equation*}
\frac{1}{2}\frac{d}{dt}\|w(t)\|^{2}=-\|w(t)\|_1^{2}-\frac{1}{2}\int_{\mathbb{S}^{1}}(w_1(t)+w_2(t)+2v(t))w(t)\partial_{x}w(t)dx
\end{equation*}
Due to Cauchy-Schwarz inequality, the injection of $H^{1}$ into $L^{\infty}$ and Young's inequality,
\begin{align*}
\int_{\mathbb{S}^{1}}\big|(w_1(t)+w_2(t)+2v(t))w(t)\partial_{x}w(t)\big|dx&\leq\|(w_1(t)+w_2(t)+2v(t))w(t)\|\cdot\|\partial_{x}w(t)\|\\
                                                                          &\leq C\big(\frac{C}{2\nu}\|w(t)\|^{2}+\frac{\nu}{2C}\|w(t)\|_{1}^{2}\big).
\end{align*}
Hence, $\frac{d}{dt}\|w(t)\|^{2}\leq C_1\|w(t)\|^{2}$. Since $w(0)=0$, Gronwall's inequality admits that  $\|w(t)\|^{2}=0$ . The uniqueness holds true.

To prove the inequality \eqref{wn}, we first discuss the case $n=0$ of \eqref{wn}.
After multiplying \eqref{DB} by $w$ and integrating in space, the left side is written as $\int_{\mathbb{S}^{1}}w\partial_twdx=\frac{1}{2}\frac{d}{dt}\|w(t)\|^{2}$, and the right side becomes
\begin{align*}
\int_{\mathbb{S}^{1}}w\big(\nu\partial_{xx}w-\frac{1}{2}\partial_{x}(w+v)^{2}\big)dx&=-\nu\|w(t)\|_{1}^{2}+\frac{1}{2}\int_{\mathbb{S}^{1}}\big(w^{2}\partial_{x}w +v^{2}\partial_{x}w+2vw\partial_{x}w\big)dx\\
                                                                                    &=-\nu\|w(t)\|_{1}^{2}+\frac{1}{2}\int_{\mathbb{S}^{1}}\big(v^{2}\partial_{x}w+2vw\partial_{x}w\big)dx.
\end{align*}
It is a direct consequence of Cauchy-Schwarz inequality, Young's inequality and \eqref{So} with $k=0$, that
\begin{equation*}
\frac{1}{2}\int_{\mathbb{S}^{1}}\big(v^{2}\partial_{x}w+2vw\partial_{x}w\big)dx\leq\frac{\nu}{4}\|w(t)\|_{1}^{2}+c\|v\|_{\mathcal{D}_T^1}^{4}+\frac{\nu}{4}\|w(t)\|_{1}^{2}+c\|v\|_{\mathcal{D}_T^1}^{2}\|w(t)\|^{2},\quad c=c(\nu).
\end{equation*}
The above yields 
\begin{equation*}
\frac{d}{dt}\|w(t)\|^{2}+\nu\|w(t)\|_{1}^{2}\leq c_1\|v\|_{\mathcal{D}_T^1}^{4}+c_2\|v\|_{\mathcal{D}_T^1}^{2}\|w(t)\|^{2}.
\end{equation*}
Gronwall's inequality infers that
\begin{equation}\label{WVT}
\|w(t)\|^{2}\leq tc_1\|v\|_{\mathcal{D}_T^1}^{4}e^{c_2\|v\|_{\mathcal{D}_T^1}^{2}}\leq c(T),\quad0\leq t\leq T,
\end{equation}
which will be especially helpful for us to estimate $\|w(t)\|_n$.

Now, we multiply the equation \eqref{DB} with $w^{(2n)}$, and use integration by parts, to obtain
\begin{align*}
&\frac{1}{2}\frac{d}{dt}\|w(t)\|_n^{2}+\nu\|w(t)\|_{n+1}^{2}\\
&\leq\Big|\Big\langle\frac{d^{n}}{dx^{n}}(w(t)+v(t))^{2},\frac{d^{n+1}}{dx^{n+1}}w(t)\Big\rangle\Big| \\
                                                           &\leq\underset{=:J_1}{\underbrace{\big|\langle w^{2}(t)^{(1)},w(t)^{(2n)}\rangle\big|}}+\underset{=:J_2}{\underbrace{2\big|\langle( w(t)v(t))^{(n)},w(t)^{(n+1)}\rangle\big|}}+\underset{=:J_3}{\underbrace{\big|\langle v^{2}(t)^{(n)},w(t)^{(n+1)}\rangle\big|}}.
\end{align*}
By \eqref{W} (with $\epsilon=\epsilon(n,2)$),
\begin{equation*}
|J_{1}|\leq C_1\|w(t)\|_{n+1}^{1+\epsilon}\|w(t)\|^{2-\epsilon},\quad\epsilon=\frac{2n+1}{2n+2}.
\end{equation*}
Applying Young's inequality to the right hand side of this inequality,
\begin{equation*}
J_{1}\leq\frac{\nu}{4}\|w(t)\|_{n+1}^{2}+C_{1}'\|w(t)\|^{c_{1}'}.
\end{equation*}
To estimate $J_{2}=2\langle w(t)v(t),\partial_{x}w(t)\rangle_{n}$, Gagliardo-Niremberg inequality and \eqref{uvn} indicate that
\begin{align*}
J_{2}&\leq C\|w(t)v(t)\|_{n}\|w(t)\|_{n+1}\leq C_1\|w(t)\|_{n}\|v(t)\|_{n}\|w(t)\|_{n+1} \\
     &\leq C_n\|v(t)\|_{n}\|w(t)\|^{1-\frac{n}{n+1}}\|w(t)\|_{n+1}^{1+\frac{n}{n+1}}.
\end{align*}
Young's inequality concludes that
\begin{equation*}
J_2\leq\frac{\nu}{4}\|w(t)\|_{n+1}^{2}+C_2'(\|v\|_{\mathcal{D}_T^n})\|w(t)\|^{c_2'}.
\end{equation*}
In a similar manner,
\begin{equation*}
J_3\leq\frac{\nu}{4}\|w(t)\|_{n+1}^{2}+C_3'(\|v\|_{\mathcal{D}_T^n})\|w(t)\|^{c_3'}.
\end{equation*}
Using \eqref{WVT}, we have
\begin{equation*}
\frac{1}{2}\frac{d}{dt}\|w(t)\|_{n}^{2}+\frac{\nu}{4}\|w(t)\|_{n+1}^{2}\leq C'(\|v(t)\|_{\mathcal{D}_T^n})C(T)^{c'}.
\end{equation*}
Finally, the integration of the above inequality on $[0,T]$ allows us to get the validity of \eqref{wn}.
\end{proof}

Consequently, Theorem \ref{eu} ensures the following result:
\begin{theorem}
For all $T>0$, $n\geq1$ and $u_0\in H^n$, there is a unique pathwise solution $u$ of model \eqref{B}.
In addition, $u$ satisfies that
\begin{equation*}
\|u\|_{\mathcal{D}_T^n}^{2}+\int_{0}^{T}\|u(t)\|_{n+1}^{2}dt\leq C_n,
\end{equation*}
where $C_n=C_n(T,\nu,\|u_0\|_{n},\|L\|_{\mathcal{D}_T^n})>0.$
\end{theorem}

\section{Oleinik-Kruzkov inequality}\label{OK}
The goal of this section is to estimate $u$ and $\partial_{x}u$. To get upper bounds for the norms, the key point is the Oleinik-Kruzkov inequality \eqref{O-K}, which we apply to solution of \eqref{B} with fixed $\omega$. The inequality was proved by Oleinik-Kruzkov for the free Burgers equation, but their
argument applies to the stochastic equation \eqref{B} in pathwise sense. The following theorem provides us with estimations on $u$ in $L^{\infty}$, and
on $\partial_{x}u$ in $L^{1}$.

\begin{theorem}\label{At}
For any initial data $u_0\in H^{1}$, any $p\geq1$, any $\epsilon\in(0,T]$ and $\nu\in(0,1]$, uniformly in $t\in[\epsilon,T]$ we have:
\begin{equation}\label{O-K}
\mathbb{E}(|u(t,\omega)|_{\infty}^{p}+|\partial_{x}u(t,\omega)|_{1}^{p})\leq C\epsilon^{-p},
\end{equation}
where the constant $C$ depends only on the random force $L(t)$. 
\end{theorem}
\begin{proof}
If $L(t)$ is not zero, we propose to take the solution $u$ of model \eqref{B} as $u=v+L$ when $u, L\in\mathcal{D}([0,T],H^{n})$. Therefore, $v:=u-L$ is solution of
\begin{equation}\label{Lv}
\partial_tv=-u\partial_xu+\nu\partial_{xx}v+\nu\partial_{xx}L(t).
\end{equation}
We calculate the derivative of the equation \eqref{Lv} with respect to $x$ and multiply it by $t$,
\begin{equation}\label{Lvxt}
t\partial_{tx}v=-t\big((\partial_xu)^{2}+u\partial_{xx}u\big)+\nu t\partial_{xxx}v+\nu t\partial_{xxx}L(t).
\end{equation}
Set $w:=t\partial_{x}v$ and rewrite the equation \eqref{Lvxt} into
\begin{equation}\label{wv}
\partial_{t}w=\partial_{x}v-t(\partial_xu)^{2}-tu\partial_{xx}L(t)-u\partial_xw+\nu\partial_{xx}w+\nu t\partial_{xxx}L(t).
\end{equation}
Now, we consider the function $w(t,x)$ on the cylinder $[0,T]\times \mathbb{S}^1$. Since $w|_{t=0}=0$ and $\int_{\mathbb{S}^1}w(x)dx=0$,
either $w$ is identically zero, or it takes maximal value $M>0$ at a point $(t_{1},x_{1})\in[0,T]\times \mathbb{S}^1$ with $t_{1}>0$.
If $w$ is identically zero, then $\partial_{x}v(t,x)\equiv0$ and $v(t,x)=0$ over $[0,T]\times \mathbb{S}^1$. Hence \eqref{O-K} holds.
If $w$ reaches its maximum $M>0$ at $(t_{1},x_{1})\in[0,T]\times \mathbb{S}^1$, we denote
\begin{equation}\label{K}
K:=\max(1,\max_{0\leq t\leq T}|L(t)|).
\end{equation}
Then,
\begin{equation*}
K\leq\max(1,C\sup_{0\leq t\leq T}\|L(t)\|_{n}),\quad C>0.
\end{equation*}

Let us show that $M\leq4TK.$
Reducing to an absurdity, suppose that $M>4TK$. Optimality conditions recognize that
\begin{equation}\label{wtx}
\partial_{t}w\geq0,\,\partial_{x}w=0\quad\text{and}\quad\partial_{xx}w\leq0\quad\text{at the point}\,(t_{1},x_{1}).
\end{equation}
By \eqref{wv} and \eqref{wtx}, we figure out
\begin{equation*}
-\partial_{x}v+t(\partial_{x}u)^{2}+tu\partial_{xx}L(t)\leq\nu t\partial_{xxx}L(t),\quad\text{at the point}\,(t_{1},x_{1}).
\end{equation*}
Multiplying this previous inequality by $t$ and using the fact that $w(t_{1},x_{1})=M$ and $t^{2}(\partial_{x}u)^{2}=(w+t\partial_{x}L_{t})^{2}$,
\begin{equation*}
-M+(M+t\partial_{x}L(t))^{2}+t^{2}u\partial_{xx}L(t)\leq\nu t^{2}\partial_{xxx}L(t),\quad\text{at point}\,(t_{1},x_{1}).
\end{equation*}
When $t\leq T$,  because of \eqref{K},
\begin{equation*}
|t\partial_{x^{j}}^{j}L_{t}|\leq TK,\quad j=0,\cdot\cdot\cdot,n.
\end{equation*}
Moreover,
\begin{equation*}
(M+t\partial_{x}L_{t})^{2}\geq(M-TK)^{2},\quad\text{at point}\,(t_{1},x_{1}),
\end{equation*}
since $M\geq TK$. In addition, $\int_{\mathbb{S}^1}tv(x)dx=0$ and $t\partial_{x}v=w\leq M$ suggest that $|tv|\leq M$, which results in
\begin{equation}\label{KM}
|tu|\leq|tL(t)|+M\leq TK+M,\quad \forall~ (t,x)\in[0,T]\times \mathbb{S}^1,
\end{equation}
and
\begin{equation*}
|t^{2}u\partial_{xx}L(t)|=|(tu)(t\partial_{xx}L(t))|\leq(TK+M)(TK),\quad \forall (t,x)\in[0,T]\times \mathbb{S}^1.
\end{equation*}
As $\nu t^{2}\partial_{xxx}L(t)\leq T^{2}K$, using \eqref{KM}, we deduce that
\begin{equation*}
-M+(M-TK)^{2}\leq (TK+M)TK+T^2K.
\end{equation*}
This last relation implies that $M<3TK$. But $M>4TK$ and $K\geq1$, we get the contraction $3TK+T<5TK$.
Hence $M\leq4TK$ is established. Since $t\partial_{x}u=w+t\partial_{x}L(t)$ and $M\leq4TK$,
for all $(t,x)\in[0,T]\times \mathbb{S}^1$, we have $t\partial_{x}u\leq5TK$. It follows that
\begin{equation*}
t|u|_{\infty},\,t|\partial_{x}u|_{1}\leq10TK\leq10T(1+\|L\|_{\mathcal{D}_T^n}),
\end{equation*}
and then $u$ satisfies \eqref{O-K}.

If $L(t)$  is zero, then we consider $w=t\partial_{x}u$, find a point where it takes maximal value
at the cylinder $[0,T]\times \mathbb{S}^1$, and write the conditions of maximality to use similar arguments as above.
\end{proof}

This very powerful estimate, jointly with stochastic partial differential equations tricks, will allow us to bound
moments for all Sobolev norms of solution $u$ from system \eqref{B} in the next section.

\section{ Moment estimates for Sobolev norms of solution}\label{ME}
Given solution $u\in\mathcal{D}([0,T];H^{n})$ of \eqref{B}, the aim of this section is to estimate the mathematical expectation of the norms $\mathbb{E}[\|u(t)\|_{n}^{2}]$, $n\geq1$, for the solution $u(t)$ uniformly in $\nu\in(0, 1]$ and in $u_0\in H^1$.
Let $f$ be a function on $H^n$, which is twice continuously
differentiable. Applying It\^{o}'s formula to compute
\begin{align}\nonumber
&\mathbb{E}[f(u(t))]=\mathbb{E}\Big[f(u_0)+\int_{0}^{t}f'(u(s))(\nu\partial_{xx}u-u\partial_{x}u)ds\\ \nonumber
&~~~+\sum_{k\in\mathbb{Z}^{\ast}}\int_{0}^{t}\int_{|y|\geq1}[f(u(s-)+y\beta_{k}e_{k})-f(u(s-))]N_{k}(ds,dy)\\ \nonumber
&~~~+\sum_{k\in\mathbb{Z}^{\ast}}\int_{0}^{t}\int_{|y|<1}[f(u(s-)+y\beta_{k}e_{k})-f(u(s-))]\tilde{N}_{k}(ds,dy)\\\label{Ito}
&~~~+\sum_{k\in\mathbb{Z}^{\ast}}\int_{0}^{t}\int_{|y|<1}\Big(f(u(s-)+y\beta_{k}e_{k})-f(u(s-))-f'(u(s-))y\beta_{k}e_{k}\Big)\mu(dy)ds\Big].
\end{align}
\begin{theorem}\label{IT}
Let $T>\epsilon>0$, $m\geq1$ and $\nu\in(0,1]$. There exists $C(n,\epsilon)>0$ such that for all $u_0\in H^1$, the solution $u$ of \eqref{B} satisfies
\begin{equation}\label{est}
\mathbb{E}[\|u(t)\|_{n}^{2}]\leq C(n,\epsilon)\nu^{-(2n-1)},\quad\forall~\,t\in[\epsilon,T].
\end{equation}
\end{theorem}
\begin{proof}
The analysis of these estimates proceeds in two steps. First, we take a function $f(u)=\|u\|_{n}^{2}=\langle(-\Delta)^{n}u,u\rangle=\langle (-\Delta)^{n/2}u,(-\Delta)^{n/2}u\rangle$ in \eqref{Ito}. Furthermore, for any $u, v\in H^{n}$,
\begin{align}\label{ff}
f'(u)v&=2\langle(-\Delta)^{n}u,v\rangle=2\langle (-\Delta)^{n/2}u, (-\Delta)^{n/2}v\rangle\leq2\|u\|_{n}\|v\|_{n},\\\label{fff}
f''(u)(v,v)&=2\langle(-\Delta)^{n}v,v\rangle=2\langle (-\Delta)^{n/2}v,(-\Delta)^{n/2}v\rangle
   =2\|v\|_{n}^{2}.
\end{align}
Rewrite \eqref{Ito} into
\begin{align}\nonumber
&\mathbb{E}[f(u(t))]\\ \nonumber
&=\mathbb{E}\Big[f(u_0)+\int_{0}^{t}2\langle(-\Delta)^{n}u(s),\nu\partial_{xx}u(s)-u(s)\partial_{x}u(s)\rangle ds\\ \nonumber
&~~~+\sum_{k\in\mathbb{Z}^{\ast}}\int_{0}^{t}\int_{|y|\geq1}[f(u(s-)+y\beta_{k}e_{k})-u(w(s-))]N_{k}(ds,dy)\\ \nonumber
&~~~+\sum_{k\in\mathbb{Z}^{\ast}}\int_{0}^{t}\int_{|y|<1}[f(u(s-)+y\beta_{k}e_{k})-f(u(s-))]\tilde{N}_{k}(ds,dy)\\ \nonumber
&~~~+\sum_{k\in\mathbb{Z}^{\ast}}\int_{0}^{t}\int_{|y|<1}\Big(f(u(s-)+y\beta_{k}e_{k})-f(u(s-))-2\langle(-\Delta)^{n}u(s-),y\beta_{k}e_{k}\rangle \Big)\mu(dy)ds\Big]\\ \label{ReIto}
&:=\mathbb{E}\big[f(u_0)+I_{1}(t)+I_{2}(t)+I_{3}(t)+I_{4}(t)\big],
\end{align}
where $2\langle(-\Delta)^{n}u(s),\nu\partial_{xx}u(s)-u(s)\partial_{x}u(s)\rangle=-2\nu\langle(-\Delta)^{n}u(s),(-\Delta)u(s)\rangle-\langle(-\Delta)^{n}u(s),\partial_{x}u^{2}(s)\rangle$ in $I_{1}(t)$. Based on Lemma \ref{XY} (with $q=\infty$) and the fact that $|u|_{\infty}\leq|\partial_xu|_1$, we obtain
\begin{equation*}
\big|\mathbb{E}[\langle(-\Delta)^{n}u(s),\partial_{x}u^{2}(s)\rangle]\big|\leq C(n)\mathbb{E}[\|u(s)\|_{n+1}^{\varepsilon+1}|\partial_{x}u(s)|_1^{2-\varepsilon}],\quad\varepsilon=\frac{2n-1}{2n+1}.
\end{equation*}
By H\"{o}lder's inequality and Theorem \ref{At}, for $\epsilon':=\frac{\epsilon}{2}\leq t\leq T$,
\begin{equation}\label{I1}
\mathbb{E}[\|u(s)\|_{n+1}^{\varepsilon+1}|\partial_{x}u(s)|_1^{2-\varepsilon}]\leq
\big(\mathbb{E}\|u(s)\|_{n+1}^{2}\big)^{\frac{2n}{2n+1}}\big(\mathbb{E}|\partial_{x}u(s)|_1^{2n+3}\big)^{\frac{1}{2n+1}}\leq C(n)\big(\mathbb{E}\|u(s)\|_{n+1}^{2}\big)^{\frac{2n}{2n+1}}.
\end{equation}
The inequality \eqref{ff} above permits us to get
\begin{align}\nonumber
\mathbb{E}\|I_{2}(t)\|_{n}&\leq C\sum_{k\in\mathbb{Z}^{\ast}}\mathbb{E}\Big(\int_{0}^{t}\int_{|y|\geq1}|f(u(s)+y\beta_{k}e_{k})-f(u(s))|N_{k}(ds,dy)\Big)\\ \nonumber
                                               &=C\sum_{k\in\mathbb{Z}^{\ast}}\mathbb{E}\Big(\int_{0}^{t}\int_{|y|\geq1}|f(u(s)+y\beta_{k}e_{k})-f(u(s))|\mu(dy)ds\Big)\\ \nonumber
                                               &\leq C\sum_{k\in\mathbb{Z}^{\ast}}\mathbb{E}\Big(\int_{0}^{t}\int_{|y|\geq1}\int_{0}^{1}\|f'(u(s)+\xi y\beta_{k}e_{k})\|_{n}d\xi\|y\beta_{k}e_{k}\|_{n}\mu(dy)ds\Big)\\ \nonumber
                                               &\leq C\sum_{k\in\mathbb{Z}^{\ast}}\mathbb{E}\Big(\int_{0}^{t}\int_{|y|\geq1}(\|u(s)\|_{n}+\| y\beta_{k}e_{k}\|_{n})\|y\beta_{k}e_{k}\|_{n}\mu(dy)ds\Big)\\ \nonumber
                                               &\leq C\sum_{k\in\mathbb{Z}^{\ast}}\beta_{k}\int_{|y|\geq1}|y|\mu(dy)\mathbb{E}\int_{0}^{t}\|u(s)\|_{n}ds+C\sum_{k=1}^{\infty}\beta_{k}^{2}\int_{|y|\geq1}|y|^{2}\mu(dy)\\ \label{I2}
                                               &\leq C\int_{0}^{t}\mathbb{E}\|u(s)\|_n^{2}ds+C.
\end{align}
Relying on H\"{o}lder's inequality, It\^{o} isometry and Young's inequality, we figure out
\begin{align}\nonumber
\mathbb{E}\|I_{3}(t)\|_{n}&\leq C\sum_{k\in\mathbb{Z}^{\ast}}\Big(\mathbb{E}\int_{0}^{t}\int_{|y|<1}|f(u(s)+y\beta_{k}e_{k})-f(u(s))|^{2}
\mu(dy)ds\Big)^{\frac{1}{2}}  \\ \nonumber
&\leq C\sum_{k\in\mathbb{Z}^{\ast}}\Big(\mathbb{E}\int_{0}^{t}\int_{|y|<1}\int_{0}^{1}\|f'(u(s)+\xi y\beta_{k}e_{k})\|_{n}^{2}d\xi\|y\beta_{k}e_{k}\|_{n}^{2}\mu(dy)ds\Big)^{\frac{1}{2}}\\ \nonumber
                                               &\leq C\sum_{k\in\mathbb{Z}^{\ast}}\Big(\mathbb{E}\int_{0}^{t}\int_{|y|<1}(\|u(s)\|_{n}^{2}+\| y\beta_{k}e_{k}\|_{n}^{2})\|y\beta_{k}e_{k}\|_{n}^{2}\mu(dy)ds\Big)^{\frac{1}{2}}\\ \nonumber
                                               &\leq C\sum_{k\in\mathbb{Z}^{\ast}}\beta_{k}\Big[\int_{|y|<1}|y|^{2}\mu(dy)\Big]^{\frac{1}{2}}\Big(\mathbb{E}\int_{0}^{t}\|u(s)\|_{n}^{2}ds\Big)^{\frac{1}{2}}+C\sum_{k\in\mathbb{Z}^{\ast} }\beta_{k}^{2}\Big[\int_{|y|<1}|y|^{4}\mu(dy)\Big]^{\frac{1}{2}}\\ \nonumber
                                               &\leq C\mathbb{E}\int_{0}^{t}\|u(s)\|_{n}^{2}ds+C\Big(\sum_{k\in\mathbb{Z}^{\ast}}\beta_{k}\Big[\int_{|y|<1}|y|^{2}\mu(dy)\Big]^{\frac{1}{2}}\Big)^{2}+C\\ \label{I3}
                                               &\leq  C\int_{0}^{t}\mathbb{E}\|u(s)\|_{n}^{2}ds+C.
\end{align}
The Taylor's expansion and \eqref{ff}-\eqref{fff}  have been successfully used to monitor
\begin{align}\nonumber
\mathbb{E}\|I_{4}(t)\|_{n}&\leq C\sum_{k\in\mathbb{Z}^{\ast}}\mathbb{E}\int_{0}^{t}\int_{|y|<1}\big|f(u(s)+y\beta_{k}e_{k})-f(u(s))-2\langle(-\Delta)^{n}u(s),y\beta_{k}e_{k}\rangle\big|\mu(dy)ds\\ \nonumber
&\leq C\sum_{k\in\mathbb{Z}^{\ast}}\mathbb{E}\int_{0}^{t}\int_{|y|<1}\|y\beta_{k}e_{k}\|_n^{2}\mu(dy)ds\\ \label{I4}
&\leq C\sum_{k\in\mathbb{Z}^{\ast}}\beta_{k}^{2}\int_{|y|<1}|y|^{2}\mu(dy)\leq C.
\end{align}

We set $X_j(t)=\mathbb{E}\|u(t)\|_j^{2}, j\in\mathbb{N}^{*}$. According to previous estimates \eqref{I1}-\eqref{I4}, from \eqref{ReIto} we obtain
\begin{equation}\label{in}
\frac{d}{dt}X_n(t)\leq CX_n(t)-2\nu X_{n+1}(t)+C(n)X_{n+1}(t)^{\frac{2n}{2n+1}},\quad \epsilon'\leq t\leq T.
\end{equation}
As before, Gagliardo-Niremberg inequality, H\"{o}lder's inequality and Theorem \ref{At} yield that
\begin{equation*}
X_n(t)\leq C(n)X_{n+1}(t)^{\frac{2n-1}{2n+1}}(\mathbb{E}|\partial_{x}u|_1^{a})^{b}\leq C(n)X_{n+1}(t)^{\frac{2n-1}{2n+1}}
\end{equation*}
for suitable constants $a, b>0$. Then
\begin{equation}\label{big}
X_{n+1}(t)\geq C(n)X_n(t)^{\frac{2n+1}{2n-1}},\quad\epsilon'\leq t\leq T.
\end{equation}
The relation \eqref{in} is rewritten into
\begin{align}\nonumber
\frac{d}{dt}X_n(t)&\leq C(n)X_{n+1}(t)^{\frac{2n-1}{2n+1}}-2\nu X_{n+1}(t)+C(n)X_{n+1}(t)^{\frac{2n}{2n+1}} \\ \nonumber
                  &\leq C(n)-2\nu X_{n+1}(t)+C(n)X_{n+1}(t)^{\frac{2n}{2n+1}}\\ \label{ins}
                   &=C(n)-X_{n+1}(t)^{\frac{2n}{2n+1}}\big(2\nu X_{n+1}(t)^{\frac{1}{2n+1}}-C(n)\big),
\end{align}
where we have used Young's inequality to get the second inequality.

In order to get \eqref{est}, we proceed by contradiction. Fix $\delta>1$ and suppose that
\begin{equation}\label{txin}
\exists~t^{*}\in(2\epsilon',T]~~\text{such that}~~X_n(t^{*})>\delta\nu^{-(2n-1)}.
\end{equation}
Let $s=t^{*}-t$, $s\in[0,t^{*}]$. So the relation \eqref{ins} is rewritten as
\begin{equation}\label{ds}
\frac{d}{ds}X_n(s)\geq-C(n)+X_{n+1}(s)^{\frac{2n}{2n+1}}\big(2\nu X_{n+1}(s)^{\frac{1}{2n+1}}-C(n)\big).
\end{equation}
If $X_n(s)>\delta\nu^{-(2n-1)}$, then by \eqref{big},
\begin{equation*}
2\nu X_{n+1}(s)^{\frac{1}{2n+1}}-C(n)\geq2\nu\big(C(n)X_n(s)^{\frac{2n+1}{2n-1}}\big)^{\frac{1}{2n+1}}-C(n)\geq  C(n)\delta^{\frac{1}{2n-1}}-C(n).
\end{equation*}
Choose $\delta_{0}\gg1$ such that $C(n)\delta^{\frac{1}{2n-1}}-C(n)>1$ for all $\delta>\delta_{0}.$
From \eqref{big} and \eqref{ds} we deduce that
\begin{align}\nonumber
\frac{d}{ds}X_n(s)&\geq-C(n)+X_{n+1}(s)^{\frac{2n}{2n+1}}\big(2\nu X_{n+1}(s)^{\frac{1}{2n+1}}-C(n)\big) \\ \label{ts}
                  &\geq-C(n)+C(n)X_{n+1}(s)^{\frac{2n}{2n-1}}(\delta^{\frac{1}{2n-1}}-1)>0,
\end{align}
where the last inequality is valid if $\delta_{0}\gg1$. Collecting together \eqref{txin} and \eqref{ts}, we get that
\begin{equation}\label{tsa}
\text{the function}~~s\mapsto X_n(s)~~\text{is increasing over}~~[0,t^{*}].
\end{equation}
For all $s\in[0,t^{*}]$, we have
\begin{align*}
\frac{d}{ds}X_n(s)&\geq-C(n)+C(n)X_{n+1}(s)^{\frac{2n}{2n-1}}(\delta^{\frac{1}{2n-1}}-1) \\
                  &\geq C(n)X_{n+1}(s)^{\frac{2n}{2n-1}}(\delta^{\frac{1}{2n-1}}-1),
\end{align*}
if $\delta>\delta_{0}\gg1$. The last inequality showcases that
\begin{equation*}
\frac{d}{ds}\big(X_n(s)^{-\frac{1}{2n-1}}\big)\leq-C(n)(\delta^{\frac{1}{2n-1}}-1).
\end{equation*}
By integrating the last relation in time between $0$ and $s$, and using \eqref{tsa},
\begin{equation*}
X_n(s)^{-\frac{1}{2n-1}}\leq-C(n)(\delta^{\frac{1}{2n-1}}-1)s+X_n(0)^{-\frac{1}{2n-1}}\leq-C(n)(\delta^{\frac{1}{2n-1}}-1)s+\delta^{-\frac{1}{2n-1}}\nu.
\end{equation*}
As $\nu\leq1$, we can find a $s'\in(0,t^{*}]$ such that $X_n(s')^{-\frac{1}{2n-1}}=0$, which results in a contradiction since $X_n(s')^{-\frac{1}{2n-1}}>0$. Therefore, \eqref{txin} is false if $\delta$ is large enough. Thus, we get \eqref{est} with $C(n,\epsilon)=\delta$.
The proof for the case $n\in\mathbb{N}^{*}$ is completed.

In the second step, we suppose that $n\geq1$ and $n\notin\mathbb{N}^{*}$. Then, there exists $j\in\mathbb{N}^{*}$ and $s\in(0,1)$ such that
$n=j+s$ and $\lceil n\rceil=j+1$. Utilizing the interpolation inequality and H\"{o}lder's inequality,
\begin{equation*}
\mathbb{E}[\|u(t)\|_{n}^{2}]\leq\mathbb{E}[\|u(t)\|_{j+1}^{2}]^{s}\mathbb{E}[\|u(t)\|_{j}^{2}]^{1-s}.
\end{equation*}
Because \eqref{est} is established for $n=j$ and $n=j+1$, the right hand term of this inequality is bounded by $
C(n,\epsilon)\nu^{-[(2(j+1)-1)s+(2j-1)(1-s)]}=C(n,\epsilon)\nu^{-(2n-1)}$.
\end{proof}

\begin{remark}
Pay attention that for $n=0$ this is wrong, and instead we have $\mathbb{E}[\|u(t)\|^{2}]\sim1$. This means that in averaging sense the solution  $u$ for \eqref{B} is of order one with $\nu\in(0,1]$.
\end{remark}
\begin{corollary}
Under the conditions of Theorem \ref{IT}, for all $k\geq1$, there exists
$C(k,n,\epsilon)>0$ such that
\begin{equation}\label{uk}
\mathbb{E}[\|u(t)\|_{n}^{k}]\leq C\nu^{-\frac{k}{2}(2n-1)},\quad \forall~ t\in[\epsilon,T].
\end{equation}
\end{corollary}
\begin{proof}
For $m>n\in\mathbb{N}^{\ast}$ and by Gagliardo-Niremberg inequality,
we have
\begin{equation*}
\mathbb{E}[\|u(t)\|_{n}^{k}]\leq C[\|u(t)\|_{m}^{k\gamma_n(m)}|u(t)|_{\infty}^{k(1-\gamma_n(m))}],\quad \gamma_n(m)=\frac{2n-1}{2m-1},
\end{equation*}
where $C=C(k,m,n)>0$. Let us choose $m$ large enough such that $k\gamma_n(m)<2$. Using
H\"{o}lder's inequality, the right term is bounded by
\begin{equation*}
C'\mathbb{E}[\|u(t)\|_{m}^{2}]^{\frac{k\gamma_n(m)}{2}}\mathbb{E}[|u(t)|_{\infty}^{a}]^{b},
\end{equation*}
with $C'(k,m,n), a(k,m,n), b(k,m,n)>0$. Hence, using \eqref{est}, we get \eqref{uk}. If $n$ is not a positive integer, then we execute similar arguments as in the proof of Theorem \ref{IT}.
\end{proof}

If we take $f(u)=\frac{1}{2}\|u\|^{2}$ in the identity \eqref{Ito}, then for $1\leq T\leq t\leq T+\sigma(\sigma>0)$ the energy balance relation for solution $u(t)$ of the stochastic equation \eqref{B} takes the
following form:
\begin{align}\nonumber
&\frac{1}{2}\mathbb{E}[\|u(T+\sigma)\|^{2}]-\frac{1}{2}\mathbb{E}[\|u(T)\|^{2}]=\mathbb{E}\Big[\int_{T}^{T+\sigma}\langle u(t),\nu\partial_{xx}u-u\partial_{x}u\rangle dt\\ \nonumber
&+\sum_{k\in\mathbb{Z}^{\ast}}\int_{T}^{T+\sigma}\int_{|y|\geq1}\frac{1}{2}(\|u(t-)+y\beta_{k}e_{k}\|^{2}-\|u(t-)\|^{2})N_{k}(dt,dy)\\ \nonumber
&+\sum_{k\in\mathbb{Z}^{\ast}}\int_{T}^{T+\sigma}\int_{|y|<1}\frac{1}{2}(\|u(t-)+y\beta_{k}e_{k}\|^{2}-\|u(t-)\|^{2})\tilde{N}_{k}(dt,dy)\\\label{fan}
&+\sum_{k\in\mathbb{Z}^{\ast}}\int_{T}^{T+\sigma}\int_{|y|<1}\frac{1}{2}\Big(\|u(t-)+y\beta_{k}e_{k}\|^{2}-\|u(t-)\|^{2}-2\langle u(t-),y\beta_{k}e_{k}\rangle \Big)\mu(dy)dt\Big].
\end{align}
We rewrite the energy balance \eqref{fan} into
\begin{align}\nonumber
&\frac{1}{2}\mathbb{E}[\|u(T+\sigma)\|^{2}]-\frac{1}{2}\mathbb{E}[\|u(T)\|^{2}]+\nu\int_{T}^{T+\sigma}\mathbb{E}[\|u(t)\|_1^{2}] dt=\\ \nonumber
&\mathbb{E}\Big[\sum_{k\in\mathbb{Z}^{\ast}}\int_{T}^{T+\sigma}\int_{|y|\geq1}\frac{1}{2}(\|u(t-)+y\beta_{k}e_{k}\|^{2}-\|u(t-)\|^{2})N_{k}(dt,dy)\\ \nonumber
&+\sum_{k\in\mathbb{Z}^{\ast}}\int_{T}^{T+\sigma}\int_{|y|<1}\frac{1}{2}(\|u(t-)+y\beta_{k}e_{k}\|^{2}-\|u(t-)\|^{2})\tilde{N}_{k}(dt,dy)\\ \label{fans}
&+\sum_{k\in\mathbb{Z}^{\ast}}\int_{T}^{T+\sigma}\int_{|y|<1}\frac{1}{2}\Big(\|u(t-)+y\beta_{k}e_{k}\|^{2}-\|u(t-)\|^{2}-2\langle u(t-),y\beta_{k}e_{k}\rangle \Big)\mu(dy)dt\Big].
\end{align}
 The term $\frac{1}{2}\mathbb{E}[\|u(t)\|^{2}]$ is called the energy
of $u(t)$, and $\mathbb{E}[\|u(t)\|_1^{2}]$ is the energy dissipation rate.
The following theorem will give us a framework for the rate of energy dissipation.
\begin{theorem}\label{Ai}
For $u(t)$ in system \eqref{B} with initial value $u_0\in H^{1}$, there exist $C_1, C_2>0$ and $\sigma_0(C_1,C_2)>0$ such that
for all $\sigma\geq\sigma_0$ and $T\geq1$:
\begin{equation*}
C_1\nu^{-1}\leq\frac{1}{\sigma}\int_{T}^{T+\sigma}\mathbb{E}[\|u(t)\|_1^{2}]dt\leq C_2\nu^{-1},
\end{equation*}
uniformly in $\nu\in(0,1]$.
\end{theorem}
\begin{proof}
Using similar arguments as in the proof of Theorem \ref{IT}, the three terms at the right side of \eqref{fans} are bounded by $C\int_{T}^{T+\sigma}\mathbb{E}[\|u(t)\|^{2}]dt+\tilde{C}$. It follows from $\frac{1}{2}\frac{d}{dt}\mathbb{E}[\|u(t)\|^{2}]\leq C\mathbb{E}[\|u(t)\|^{2}]$ and Gronwall's inequality that $\mathbb{E}[\|u(t)\|^{2}]$ is bounded by a constant which depends only on the random force.
Hence we get the result by utilizing \eqref{fans} again.
\end{proof}

For any random function $t\mapsto R(t,\omega)$ (i.e., for a random process $R$), we denote by $\langle\langle R\rangle\rangle$ its
averaging in ensemble and local averaging in time,
\begin{equation*}
\langle\langle R\rangle\rangle=\frac{1}{\sigma}\int_{T}^{T+\sigma}\mathbb{E}[R(t,\omega)]dt,
\end{equation*}
where $T\geq1$ and $\sigma\geq\sigma_0>0$ are parameters.
In this notation, the inequality in Theorem \ref{Ai} that we have just proved, has the expression
\begin{equation*}
C_{1}\nu^{-1}\leq\langle\langle \|u\|_{1}^{2}\rangle\rangle\frac{1}{2}\leq C_2\nu^{-1}.
\end{equation*}
So,
\begin{equation*}
\langle\langle \|\partial_xu\|_{L^{2}}^{2}\rangle\rangle=\langle\langle \|u^{\nu}\|_{1}^{2}\rangle\rangle\sim\nu^{-1},
\end{equation*}
where $\sim$ means that the ratio of two quantities is bounded from below and from above, uniformly in $\nu$ and in $T\geq1$ and $\sigma\geq\sigma_0$, entering the brackets $\langle\langle\cdot\rangle\rangle$.

Now, let's show the basic estimate for Sobolev norms of solution : $\langle\langle \|u\|_{n}^{2}\rangle\rangle\sim\nu^{-(2n-1)}$.
\begin{theorem}\label{TH}
Let $n\in\mathbb{N}^{\ast}$, $\sigma\geq\sigma_0>0$ and $T\geq1$.
For any $u_0\in H^{1}$, there exists $C_n(\sigma_0)>1$ such that the solution $u$ of \eqref{B} satisfies
\begin{equation}\label{hua}
C_n^{-1}\nu^{-(2n-1)}\leq\langle\langle \|u\|_{n}^{2}\rangle\rangle\leq C_n\nu^{-(2n-1)},
\end{equation}
uniformly in $\nu\in(0,1]$.
\end{theorem}
\begin{proof}
The upper bound in the right inequality of \eqref{hua} follows from Theorem \ref{IT}. We have already obtained the lower bound for the averaged first Sobolev norm in Theorem \ref{Ai}. So, it remains to prove the lower bound for $\langle\langle \|u\|_{n}^{2}\rangle\rangle$ when $n\in\mathbb{N}^{\ast}$ and $n\geq2$. By an application of Gagliardo-Nirenberg interpolation inequality,
\begin{equation*}
 \|\partial_{x}u\|^{2}\leq c\|\partial_{x}u\|_{n-1}^{2}|\partial_{x}u|_{1}^{\frac{2n-2}{2n-1}},\quad c>0.
\end{equation*}
Apply H\"{o}lder inequality to the integral $\sigma\int_{\Omega}\int_{T}^{T+\sigma}\cdot\cdot\cdot dt\mathbb{P}(d\omega)$ and use Theorem \ref{At},
\begin{equation*}
\langle\langle \|u\|_{1}^{2}\rangle\rangle\leq c\langle\langle \|u\|_{n}^{2}\rangle\rangle^{\frac{1}{2n-1}}\langle\langle |\partial_{x}u|_{1}^{2}\rangle\rangle^{\frac{2n-2}{2n-1}}\leq C\langle\langle \|u\|_{n}^{2}\rangle\rangle^{\frac{1}{2n-1}},
\end{equation*}
that is to say,
\begin{equation*}
\langle\langle \|u\|_{n}^{2}\rangle\rangle\geq C^{1-2n}\langle\langle \|u\|_{1}^{2}\rangle\rangle^{2n-1}.
\end{equation*}
Combining this with $\langle\langle \|u\|_{1}^{2}\rangle\rangle\geq C\nu^{-1}$, we get the lower bound for $\langle\langle \|u\|_{n}^{2}\rangle\rangle$:
\begin{equation*}
\langle\langle \|u^{\nu}\|_{n}^{2}\rangle\rangle\geq C_{n}^{-1}\nu^{-(2n-1)},\quad\forall\,n\in\mathbb{N}^{\ast}.
\end{equation*}
Exactly, it's the left inequality of \eqref{hua}.
\end{proof}
This theorem turns out to be a powerful and efficient tool to
study stochastic turbulence in the one-dimensional Burgers equation \eqref{B}.
\begin{corollary}
For $n\in\mathbb{N}^{\ast}$ and $k\geq1$, there exists $C(k,n,\sigma_0)>0$ such that
\begin{equation}\label{UK}
C^{-1}\nu^{-n+\frac{1}{2}}\leq\langle\langle \|u\|_{n}^{k}\rangle\rangle^{\frac{1}{k}}\leq C\nu^{-n+\frac{1}{2}}.
\end{equation}
\end{corollary}
\begin{proof}
After averaging in \eqref{uk}, the right inequality in \eqref{UK} is immediate. If $k\geq2$, then the left inequality of
\eqref{UK}  is a result of H\"{o}lder's inequality and \eqref{hua} occurring earlier. Now, we only need to establish the left inequality of
\eqref{UK} for $k\in[1,2)$.  Taking advantage of  H\"{o}lder's inequality, we exploit
\begin{equation*}
\langle\langle \|u\|_{n}^{2}\rangle\rangle=\langle\langle \|u\|_{n}^{\frac{4}{3}}\|u\|_{n}^{\frac{2}{3}}\rangle\rangle\leq\langle\langle \|u\|_{n}^{4}\rangle\rangle^{\frac{1}{3}}\langle\langle \|u\|_{n}\rangle\rangle^{\frac{2}{3}}.
\end{equation*}
Make use of \eqref{UK} with $k=2$ and $k=4$,
\begin{equation*}
\langle\langle \|u\|_{n}\rangle\rangle\geq\langle\langle \|u\|_{n}^{2}\rangle\rangle^{\frac{3}{2}}\langle\langle \|u\|_{n}^{4}\rangle\rangle^{-\frac{1}{2}}\geq\big(C^{-1}(2,n,\sigma_{0})\nu^{-n+\frac{1}{2}}\big)^{3}\big(C(4,n,\sigma_{0})\nu^{-n+\frac{1}{2}}\big)^{-2}=:C^{-1}\nu^{-n+\frac{1}{2}},
\end{equation*}
and then the left inequality of \eqref{UK} is established for $k=1$.
Finally, for $k\in(1,2)$, the left inequality in \eqref{UK} is a consequence of
that with $k=1$ and H\"{o}lder's inequality.
\end{proof}

\section{Stochastic turbulence}\label{ST}
Our aim in the present section is to study the statistical quantities of one-dimensional turbulence $u(t,x)$ given by stochatic Burgers equation with respect to cylindrical L\'evy processes.
\subsection{The structure function for stochastic turbulence}
The structure function is one of the main objects of hydrodynamic turbulence \cite{MW}. For the one-dimensional fluid described by stochastic Burgers equation, the structure function is defined as follows:

\begin{definition} Small-scale increments corresponding to the solution $u$ of stochastic Burgers equation \eqref{B} are $|u(x+l)-u(x)|$, $x\in \mathbb{S}^{1}$, $|l|\ll1$. Their moments of degree $p>0$ are
\begin{equation*}
\langle\langle \int_{\mathbb{S}^{1}} |u(x+l)-u(x)|^{p} dx\rangle\rangle=:S_p(l; u),
\end{equation*}
where $(l,p)\mapsto S_p(l; u)$ is called the structure function of $u$.
\end{definition}
In physics, the basic quantity characterising a solution $u(t,x)$ as a one-dimensional
turbulent flow is its dissipation scale $l_d$, also known as Kolmogorov's inner scale. The dissipation scale for Burgulence described by \eqref{B} in the Fourier presentation is $l_d=C\nu^{-1}$, such that for $|k|\geq l_d$ the
averaged squared norm of the $k$-th Fourier coefficient $\widehat{u}_{k}(t)$ decays very fast,
where $\widehat{u}_{k}(t)$ is from Fourier series $u(t,x)=\sum_{k=\pm1,\pm2,\cdot\cdot\cdot}\widehat{u}_{k}(t)e^{2i\pi kx}$. In other words, for any $N\in\mathbb{N}^{\ast}$ and $\gamma>0$ there exists a $C_{N,\gamma}$ such that
\begin{equation}\label{PA}
\langle\langle |\widehat{u}_{k}(t)|^{2}\rangle\rangle\leq C_{N,\gamma}|k|^{-N},\quad\forall\,|k|\geq\nu^{-1-\gamma}.
\end{equation}
To check this, we know that $\langle\langle |\widehat{u}_{k}(t)|^{2}\rangle\rangle\leq C_n|k|^{-2n}\nu^{-(2n-1)}\leq C_n|k|^{-2n\frac{\gamma}{1+\gamma}}$, $n\in\mathbb{N}^{\ast}$, by using \eqref{hua} in Theorem \ref{TH}.

The turbulence ranges are zones specifying the size for increments of $x$. The dissipation range, the inertial range and the energy range in $x$
are non-empty and non-intersecting intervals $(0,\hat{C}_{1}\nu]=(0,l_{d}^{-1}]$, $(\hat{C}_{1}\nu,\hat{C}_2]=(l_{d}^{-1}, \hat{C}_2]$ and $(\hat{C}_2,1]$, respectively.
Here, $\hat{C}_1,\hat{C}_2>0$ depend on the random force.

The functions $S_p(l;u)$ satisfy the following upper estimates:
\begin{lemma}\label{Ma}
For $|l|\in(0,1]$, $\nu\in(0,1]$ and $p>0$, there is $C_p>0$ such that
\begin{eqnarray*}
S_p(l; u)\leq\left\{\begin{array}{l}
C_p|l|^{p}\nu^{-(p-1)},\quad\text{if}~~p\geq1;\\
C_p|l|^{p},\quad\quad\quad\quad\text{if}~~p\in(0,1).
\end{array}
\right.
\end{eqnarray*}
\end{lemma}
\begin{proof}
If we begin by considering the case $p\geq1$, then
\begin{equation*}
S_p(l; u)\leq\langle\langle \int_{\mathbb{S}^{1}} |u(x+l)-u(x)|^{p} dx\cdot\max_{x}|u(x+l)-u(x)|^{p-1}\rangle\rangle.
\end{equation*}
By H\"{o}lder's inequality, we have
\begin{equation*}
S_p(l; u)\leq\underset{=:I}{\underbrace{\langle\langle\Big(\int_{\mathbb{S}^{1}} |u(x+l)-u(x)|^{p}dx\Big)^{p}\rangle\rangle^{\frac{1}{p}}}}\underset{=:J}{\underbrace{\langle\langle \max_{x}|u(x+l)-u(x)|^{p}\rangle\rangle^{\frac{p-1}{p}}}}.
\end{equation*}
On one hand,  noticing that  the space average of $x\mapsto u(x+l)-u(x)$ vanishes identically for
all $t$. We have
\begin{align*}
\int_{\mathbb{S}^{1}}|u(x+l)-u(x)|dx&\leq\int_{\mathbb{S}^{1}}(u(x+l)-u(x))^{+}dx+\int_{\mathbb{S}^{1}}(u(x+l)-u(x))^{-}dx\\
  &\leq2\int_{\mathbb{S}^{1}}(u(x+l)-u(x))^{+}dx\leq2\sup_{x}(\partial_{x}u)^{+}\cdot |l|,
\end{align*}
which yields that $I\leq2|l|\langle\langle[\sup_{x}(\partial_{x}u)^{+}]^{p}\rangle\rangle^{\frac{1}{p}}$. So, by Theorem \ref{At}, we get that $I\leq C_{p}|l|$.
On the other hand, $J\leq\langle\langle |l|^{p}|\partial_{x}u|_{\infty}^{p}\rangle\rangle^{\frac{p-1}{p}}$.
From Gagliardo-Nirenberg interpolation inequality and H\"{o}lder's inequality, we obtain
\begin{align*}
\langle\langle |l|^{p}|\partial_{x}u|_{\infty}^{p}\rangle\rangle^{\frac{p-1}{p}}&\leq\Big(C|l|^{p}\langle\langle\|u\|_{n}^{\frac{2p}{2n-1}}|\partial_{x}u|_{1}^{\frac{(2n-3)p}{2n-1}}\rangle\rangle\Big)^{\frac{p-1}{p}} \\
&\leq C_pl^{p-1}\langle\langle\|u\|_{n}^{2}\rangle\rangle^{\frac{p-1}{2n-1}}\langle\langle|\partial_{x}u|_{1}^{\frac{(2n-3)p}{2n-1-p}}\rangle\rangle^{\frac{(2n-1-p)(p-1)}{(2n-1)p}}.
\end{align*}
Using Theorem \ref{At} and \ref{TH}, we get that $J\leq C_p|l|^{p-1}\nu^{-(p-1)}$.
Finally, $S_{p}(l; u)\leq IJ\leq C_p|l|^{p}\nu^{-(p-1)}$.

The case $p\in(0,1)$ follows immediately from the case $p=1$ and H\"{o}lder's inequality:
\begin{equation*}
S_p(l; u)\leq\langle\langle\int_{\mathbb{S}^{1}}|u(x+l)-u(x)|dx\rangle\rangle^{p}=S_1(l; u)^{p}\leq C_p|l|^p.
\end{equation*}
\end{proof}
For $|l|\in(\hat{C}_1\nu,1]$, we have a better upper bound if $p\geq1$.
\begin{lemma}
For $\nu\in(0,1]$, $|l|\in(\hat{C}_1\nu,1]$ and $p>0$, there is $C_p>0$ such that
\begin{eqnarray*}
S_p(l; u)\leq\left\{\begin{array}{l}
C_p|l|,\quad\text{if}~~p\geq1;\\
C_p|l|^{p},~~~\text{if}~~p\in(0,1).
\end{array}
\right.
\end{eqnarray*}
\end{lemma}
\begin{proof}
The calculations are almost the same as in Lemma \ref{Ma}. The only
difference is that we use another bound for $J$, i.e.,
\begin{equation*}
S_{p}(l; u)\leq C_{p}|l|J\leq C_{p}|l|\langle\langle (2|u|_{\infty})^{p}\rangle\rangle^{\frac{p-1}{p}}\leq C_{p}|l|.
\end{equation*}
\end{proof}
Now, we prove the lower estimates for $S_{p}(l; u)$.
\begin{lemma}\label{Xiao}
Assume that $\nu\in(0,\nu_{0}]$, where the constant $\nu_{0}\in(0,1]$ only depends on random force with $0<\hat{C}_1\nu_{0}<\hat{C}_2<1$. For $|l|\in(\hat{C}_1\nu,\hat{C}_2]$ and $p>0$, there is $C_p>0$ such that
\begin{eqnarray}\label{Cool}
S_{p}(l; u)\geq\left\{\begin{array}{l}
C_p|l|,\quad\text{if}~~p\geq1;\\
C_p|l|^{p},~~~\text{if}~~p\in(0,1).
\end{array}
\right.
\end{eqnarray}
\end{lemma}
\begin{proof}
Define the probability space
\begin{equation*}
(\Omega_{T},\mathcal{F}_{T},\rho):=\Big([T,T+\sigma]\times\Omega,\mathcal{T}\times\mathcal{F},\frac{dt}{\sigma}\times\mathbb{P}\Big),
\end{equation*}
where $\sigma\geq\sigma_0>0$, $T\geq1$ and $\mathcal{T}$ is the Borel $\sigma$-algebra on $[T,T+\sigma]$.
Let $\varepsilon>0$ and $Q_{1}=\{(t,\omega)\in\Omega_{T}: \|u(t,\omega)\|_{1}\leq\varepsilon\}$.
Then $\rho(Q_{1})\geq C(1,\sigma_{0})$. Let $K>0$ and
\begin{equation*}
Q_{2}=\{(t,\omega)\in Q_{1}: |\partial_{x}u^{+}(t,\omega)|_{\infty}+|\partial_{x}u(t,\omega)|_{1}+\nu^{\frac{3}{2}}\|u(t,\omega)\|_{2}+\nu^{\frac{5}{2}}\|u(t,\omega)\|_{3}\leq K\}.
\end{equation*}
By Theorem \ref{At}, estimate \eqref{UK} and Chebyshev's inequality:
\begin{equation*}
\rho(Q_{2})\geq C(1,\sigma_{0})-C_{1}K^{-1}\geq\frac{1}{2}C(1,\sigma_{0}),
\end{equation*}
for all $\nu\in(0,\nu_{0}]$ and if $K$ is sufficiently large. Let $(t,\omega)\in Q_{2}$ and denote $u(t,\omega,x)$ by $u(x)$. To establish \eqref{Cool}, we show that $u$ satisfies
\begin{equation}\label{ha}
\int_{\mathbb{S}^{1}}|u(x+l)-u(x)|^{p}dx\geq C|l|^{\min(1,p)},\quad |l|\in[\hat{C}_{1}\nu,\hat{C}_{2}],~~p>0,
\end{equation}
uniformly in $\nu\in(0,\nu_{0}]$, where $C=C(\hat{C}_1,\hat{C}_2,p)>0$.

First, consider the case $p\geq1$. Note that
\begin{equation*}
C \nu^{-1}\leq\int_{\mathbb{S}^{1}}|\partial_{x}u|^{2}dx\leq|\partial_{x}u|_{\infty}|\partial_{x}u|_{1}\leq K|\partial_{x}u|_{\infty}.
\end{equation*}
Therefore,
\begin{equation}\label{haha}
|\partial_{x}u|_{\infty}\geq C K^{-1}\nu^{-1}=:\tilde{C}\nu^{-1}.
\end{equation}
 As $|\partial_{x}u^{+}|_{\infty}\leq K$, we gain $|\partial_{x}u^{+}|_{\infty}\leq\frac{1}{2}\tilde{C}\nu^{-1}$ if $\nu\leq\frac{1}{2}\tilde{C}K^{-1}=:\nu_{0}$. Thus, by \eqref{haha}, we have
\begin{equation*}
|\partial_{x}u^{+}|_{\infty}\leq\frac{1}{2}\tilde{C}\nu^{-1}~~\text{and}~~|\partial_{x}u^{-}|_{\infty}\geq\tilde{C}\nu^{-1},\quad \text{if}~~\nu\in(0,\nu_{0}].
\end{equation*}
Denoted by $y=y(t,\omega)=\min\{x\in[0,1): \partial_{x}u^{-}(x)\geq\tilde{C}\nu^{-1}\}$: $y$ is a well-defined measurable function over $Q_{2}$ if $\nu\in(0,\nu_{0}]$. Admittedly,
\begin{equation}\label{hei}
\int_{\mathbb{S}^{1}}|u(x+l)-u(x)|^{p}dx\geq\int_{y-\frac{|l|}{2}}^{y}\Big|\int_{x}^{x+l}\partial_{x}u^{-}(z)dz-\int_{x}^{x+l}\partial_{x}u^{+}(z)dz\Big|^{p}dx.
\end{equation}
Gagliardo-Niremberg inequality ensures that $|\partial_{xx}u|_{\infty}\leq c\|u\|_{2}^{\frac{1}{2}}\|u\|_{3}^{\frac{1}{2}}$, which implies that $|\partial_{xx}u|_{\infty}\leq cK\nu^{-2}$. So in the interval $[x,x+\tilde{c}\nu]$, $\tilde{c}>0$, we have
\begin{equation*}
\partial_{x}u^{-}\geq\tilde{C}\nu^{-1}-\tilde{c}cK\nu^{-1}=\frac{3}{4}\tilde{C}\nu^{-1},\quad \text{if}~~\tilde{c}=\frac{\tilde{C}}{4cK}.
\end{equation*}
Suppose that $|l|\geq\tilde{c}\nu$. Because $\partial_{x}u^{+}\leq K$,
\begin{equation*}
\int_{x}^{x+l}\partial_{x}u^{-}(z)dz\geq\int_{x}^{x+\tilde{c}\nu}\partial_{x}u^{-}(z)dz\geq\frac{3}{4}\tilde{C}\tilde{c},\quad\text{and}~~\int_{x}^{x+l}\partial_{x}u^{+}(z)dz\leq Kl.
\end{equation*}
Using \eqref{hei} we get that
\begin{equation*}
\int_{\mathbb{S}^{1}}|u(x+l)-u(x)|^{p}dx\geq\int_{y-\frac{|l|}{2}}^{y}|\frac{3}{4}\tilde{C}\tilde{c}-Kl|^{p}dx\geq\frac{|l|}{2}(\frac{1}{2}\tilde{C}\tilde{c})^{p},
\end{equation*}
provided that $|l|\in[\tilde{c}\nu,\frac{\tilde{C}\tilde{c}}{4K}]$ and $\nu\in(0,\nu_{0}]$. Thus, the inequality \eqref{Cool} is established with $\nu_{0}=\frac{1}{2}\tilde{C}K^{-1}$, $\hat{C}_1=\tilde{c}$ and $\hat{C}_2=\frac{\tilde{C}\tilde{c}}{4K}$, if $p\geq1$.

Now, suppose that $p\in(0,1)$. Let $f$ be a positive arbitrary function. We can write it as $f=f^{\frac{2(1-p)}{2-p}}f^{\frac{p}{2-p}}$. So, by means of H\"{o}lder's inequality, we have $(\int f)^{2-p}\leq(\int f^2)^{1-p}(\int f^p)$. Hence,
\begin{align*}
\int_{\mathbb{S}^{1}}|u(x+l)-u(x)|^{p}dx&\int_{\mathbb{S}^{1}}\big(\big[u(x+l)-u(x)\big]^{+}\big)^{p}dx\\
                                        &\geq\Big(\int_{\mathbb{S}^{1}}\big(\big[u(x+l)-u(x)\big]^{+}\big)^{2}dx\Big)^{p-1}\Big(\int_{\mathbb{S}^{1}}\big(\big[u(x+l)-u(x)\big]^{+}\big)dx\Big)^{2-p}.
\end{align*}
Owing to $\partial_{x}u^{+}\leq K$, we observe that $\big[u(x+l)-u(x)\big]^{+}\leq M|l|$. Moreover, $p-1<0$, so the first term of the right hand side of this last inequality is reduced to $(K^{2}|l|^{2})^{p-1}$. Observe that $\int_{\mathbb{S}^{1}}[u(x+l)-u(x)]dx=0$. Therefore
\begin{equation*}
\int_{\mathbb{S}^{1}}\big[u(x+l)-u(x)\big]^{+}dx=\frac{1}{2}\int_{\mathbb{S}^{1}}|u(x+l)-u(x)|dx,
\end{equation*}
and utilizing \eqref{ha} with $p=1$, we get that the second term is reduced to $C|l|^{2-p}$. Finally, \eqref{Cool} is established for the case $p\in(0,1)$.
\end{proof}
In the same way as above, we have the following lower estimates for $S_{p}(l; u)$.
\begin{lemma}\label{Smile}
Assume that $\nu\in(0,\nu_{0}]$, where the constant $\nu_{0}\in(0,1]$ only depends on random force with $0<\hat{C}_1\nu_{0}<\hat{C}_2<1$. For $|l|\in(0,\hat{C}_1\nu]$ and $p>0$, there is $C_p>0$ such that
\begin{eqnarray*}
S_p(l; u)\geq\left\{\begin{array}{l}
C_p|l|^{p}\nu^{-(p-1)},\quad\text{if}~~p\geq1;\\
C_p|l|^{p},\quad\quad\quad\quad\text{if}~~p\in(0,1).
\end{array}
\right.
\end{eqnarray*}
\end{lemma}
\begin{proof}
The computations are almost the same as in Lemma \ref{Xiao}. The only
difference is that in the case $p\geq1$, for $|l|\leq \hat{C}_1\nu$, by H\"{o}lder's inequality we get
\begin{align*}
\int_{\mathbb{S}^{1}}|u(x+l)-u(x)|^{p}dx&\geq\int_{y-\hat{C}_1\nu}^{y+\hat{C}_1\nu}|u(x+l)-u(x)|^{p}dx  \\
                                        &\geq(2\hat{C}_1\nu)^{-(p-1)}\Big(\int_{y-\hat{C}_1\nu}^{y+\hat{C}_1\nu}|u(x+l)-u(x)|dx\Big)^{p}\\
                                        &\geq C_p\nu^{-(p-1)}\Big(\int_{y-\hat{C}_1\nu}^{y+\hat{C}_1\nu}\int_{x}^{x+l}[\partial_{x}u^{-}(z)-\partial_{x}u^{+}(z)]dzdx\Big)^{p}\\
                                        &\geq C_p\nu^{-(p-1)}\Big(\int_{y-\hat{C}_1\nu}^{y+\hat{C}_1\nu}C|l|\nu^{-1}dx\Big)^{p}\geq C_p|l|^{p}\nu^{-(p-1)}.
\end{align*}
\end{proof}

Summing up the results of Lemma \ref{Ma}-\ref{Smile} above, we obtain the following theorem.
\begin{theorem}\label{Wy}
For $|l|$ in the inertial range $(\hat{C}_{1}\nu,\hat{C}_2]$ we have
\begin{equation*}
S_p(l; u)\sim|l|^{\min(1,p)},\quad\text{where}~~p>0.
\end{equation*}
While for $|l|$ in the dissipation range $(0,\hat{C}_{1}\nu]$,
\begin{equation*}
S_p(l; u)\sim|l|^{p}\nu^{1-\max(1,p)},\quad\text{where}~~p>0.
\end{equation*}
\end{theorem}
\begin{remark}
In K41 theory the hydrodynamical dissipative scale is predicted to be $l_d^{K}=\nu^{-\frac{3}{4}}$. For water turbulence the K41 theory predicts that in the inertial range
\begin{equation*}
S_p(l; u):=\mathbb{E}|u(x+l)-u(x)|^{p}\sim|l|^{\frac{p}{3}},\quad|l|\in[\hat{C}_{1}\nu^{\frac{3}{4}},\hat{C}_{2}],
\mathrm{}\end{equation*}
where $u$ is a homogeneous random field.
This is the celebrated $\frac{1}{3}$-law of the K41 theory about the $p$-th moment of the random variable $u(x+l)-u(x)$. It claims that the sizes of
increments $|u(x+l)-u(x)|$ behaves as $|l|^{\frac{1}{3}}$ for $|l|$ in the inertial range.
The $\frac{1}{3}$-law tells us that
\begin{equation*}
\frac{S_p(l; u)^{\frac{1}{p}}}{S_p(l; u)^{\frac{1}{q}}}\sim C_{p,q}\quad\forall p, q>0,
\end{equation*}
even for tiny $|l|$, when $u(x+l)-u(x)$ is a Gaussian random
variable (very small).
\end{remark}
\begin{remark}
In our case, $u(x+l)-u(x)$ certainly is a non-Gaussian random
variable. The structure functions $S_{p}(l; u)$ obey the law in Theorem \ref{Wy} that presents an abnormal scaling.  For stochastic turbulence,
\begin{equation*}
\frac{S_p(l; u)^{\frac{1}{p}}}{S_p(l; u)^{\frac{1}{q}}}\sim C_{p,q}|l|^{\frac{1}{p}-\frac{1}{q}},
\end{equation*}
which is big for small $l$ if $p>q$. This is a typical non-Gaussian behavior.
\end{remark}

\subsection{The energy spectrum for stochastic turbulence}
The second celebrated law for the Kolmogorov theory of turbulence deals with the distribution of the
energy $\langle\langle\frac{1}{2}\int_{\mathbb{S}^{1}}|u|^{2}dx\rangle\rangle$ along the spectrum \cite{AFLV}.
For one-dimensional turbulent flow $u(t,x)$ in stochastic model \eqref{B}, by Parseval's identity,
\begin{equation*}
\langle\langle\frac{1}{2}\int_{\mathbb{S}^{1}}|u|^{2}dx\rangle\rangle=\sum_{k\in\mathbb{Z}^{\ast}}\frac{1}{2}\langle\langle|\widehat{u}_{k}|^{2}\rangle\rangle.
\end{equation*}
So we consider the quantities $\frac{1}{2}\langle\langle|\widehat{u}_{k}|^{2}\rangle\rangle$.
For any $n\in\mathbb{N}^{*}$, define $E_n(u)$ as the averaging of $\frac{1}{2}\langle\langle|\widehat{u}_{k}|^{2}\rangle\rangle$ along the layer $J_{n}=\{k\in\mathbb{Z}^{\ast}: M^{-1}n\leq|k|\leq Mn\}$ around $n$, i.e.,
\begin{equation*}
E_n(u)=\langle\langle e_n(u)\rangle\rangle,\quad e_n(u)=\frac{1}{|J_{n}|}\sum_{k\in J_{n}}\frac{1}{2}|\widehat{u}_{k}|^{2};
\end{equation*}
where $e_n(u)$ is the averaged energy of the $n$-th mode of $u$. The function $n\mapsto E_n(u)$ is called the energy spectrum  for the flow $u$.

Equivalently,
\begin{definition}
For the energy of wave number $n$ corresponding to the solution $u$ of stochastic Burgers equation \eqref{B}, the function $n\mapsto E_n(u)$ satisfying
\begin{equation}\label{En}
E_{n}(u)=\frac{1}{2n(M-M^{-1})}\sum_{M^{-1}n\leq|k|\leq Mn}\frac{1}{2}\langle\langle|\widehat{u}_{k}|^{2}\rangle\rangle
\end{equation}
is the layer-averaged energy spectrum, where $M$ is a positive constant independent of $\nu$.
\end{definition}

The estimate \eqref{PA} of $\langle\langle |\widehat{u}_{k}|^{2}\rangle\rangle$  infers that if $k$ is greater than the critical threshold $\nu^{-1}$, then  it decreases faster than any negative power of $k$, and that this is not valid if $k\ll\nu$.
It follows that for
$n\gg l_{d}$ the energy spectrum decays faster than any negative degree
of $n$ uniformly in $\nu$. But for $n\leq l_{d}$ the behaviour of $E_n(u)$ is quite different.

In what follows, we shall continue to the study the energy spectrum $E_n(u)$ when $n\lesssim\nu^{-1}$.

\begin{theorem}\label{O}
Let $M\geq1$ in \eqref{En} be large enough, for $n^{-1}$ in the inertial range $(\hat{C}_{1}\nu,\hat{C}_2]$ as in Theorem \ref{Wy}, i.e.,
\begin{equation}\label{Enen}
\hat{C}_2^{-1}\leq n<\hat{C}_{1}^{-1}\nu^{-1},
\end{equation}
we have the spectral power law
\begin{equation}\label{Ee}
\hat{C}_{3}n^{-2}\leq E_{n}(u)\leq \hat{C}_4n^{-2},
\end{equation}
which means that $E_{n}(u)\sim n^{-2}$.
\end{theorem}
\begin{proof}
Since $\widehat{u}_{k}(t)=\int_{\mathbb{S}^{1}}u(t,x)e^{-2i\pi kx}dx$, after integration by parts, we know that $|\widehat{u}_{k}|\leq\frac{1}{2\pi k}|\partial_xu|_{1}$, $k\in\mathbb{N}^{*}$. By Theorem \ref{At} and the meaning of averaging $\langle\langle\cdot\rangle\rangle$, we obtain that $\langle\langle|\widehat{u}_{k}|^{2}\rangle\rangle\leq Ck^{-2}$, which results in the upper estimate of \eqref{Ee}. Now we check the lower estimate. As $\langle\langle|\widehat{u}_k|^{2}\rangle\rangle\leq Ck^{-2}$,
\begin{equation}\label{Mk}
\sum_{|k|\leq M^{-1}n}|k|^{2}\langle\langle|\widehat{u}_k|^{2}\rangle\rangle\leq CM^{-1}n,
\end{equation}
and
\begin{equation}\label{nMk}
\sum_{|k|\geq Mn}\langle\langle|\widehat{u}_k|^{2}\rangle\rangle\leq CM^{-1}n^{-1}.
\end{equation}
Let's pose $U=\sum_{|k|\leq Mn}|k|^{2}\langle\langle|\widehat{u}_k|^{2}\rangle\rangle$. Using the fact $|\sin(x)|\leq|x|$,
\begin{align}\nonumber
U&\geq\frac{n^{2}}{\pi^{2}}\sum_{|k|\leq Mn}\sin^{2}(\frac{k\pi}{n})\langle\langle|\widehat{u}_k|^{2}\rangle\rangle \\ \label{U}
 &=\frac{n^{2}}{\pi^{2}}\Big(\sum_{k\in\mathbb{Z}^{\ast}}\sin^{2}(\frac{k\pi}{n})\langle\langle|\widehat{u}_k|^{2}\rangle\rangle-\sum_{|k|> Mn}\sin^{2}(\frac{k\pi}{n})\langle\langle|\widehat{u}_k|^{2}\rangle\rangle\Big).
\end{align}
Note that by Parseval's identity, we get
\begin{equation*}
\|u(\cdot+y)-u(\cdot)\|^{2}=4\sum_{k\in\mathbb{Z}^{\ast}}\sin^{2}(k\pi y)|\widehat{u}_k|^{2}.
\end{equation*}
So \eqref{nMk} and \eqref{U} imply that
\begin{equation}\label{US}
U\geq\frac{n^{2}}{\pi^{2}}\Big(\frac{1}{4}\langle\langle\|u(\cdot+\frac{1}{n})-u(\cdot)\|^{2}\rangle\rangle-\sum_{|k|>Mn}\langle\langle|\widehat{u}_k|^{2}\rangle\rangle\Big)\geq Cn^{2}S_{2}(\frac{1}{n})-C'M^{-1}n.
\end{equation}
Because $n$ satisfies \eqref{Enen}, by \eqref{US} and \eqref{Cool} ($p=2$, $l=\frac{1}{k}$), we obtain
\begin{equation}\label{UC}
U\geq C''n^{2}n^{-1}-C'M^{-1}n=(C''-C'M^{-1})n.
\end{equation}
Arguably,
\begin{equation*}
E_{n}(u)\geq \frac{1}{4M^{3}n^{3}}\sum_{M^{-1}n\leq|k|\leq Mn}|k|^{2}\langle\langle|\widehat{u}_k|^{2}\rangle\rangle.
\end{equation*}
Hence, utilizing \eqref{Mk} and \eqref{UC}, we specify that
\begin{align*}
E_{n}(u)&\geq\frac{1}{4M^{3}n^{3}}\Big(U-\sum_{|k|\leq M^{-1}n}|k|^{2}\langle\langle|\widehat{u}_k|^{2}\rangle\rangle \Big)\\
     &\geq\frac{C''-C'M^{-1}-CM^{-1}}{4M^{3}n^{2}}>\hat{C}_{3}n^{-2},\quad\text{with}~~\hat{C}_{3}>0,
\end{align*}
if we choose $M\gg1$. Therefore, the first inequality of \eqref{Ee} holds.
\end{proof}
Ignoring the multiplicative constants before the powers of $\nu$ in the Fourier presentation, we write the segment $[\hat{C}_{2}^{-1},\hat{C}_{1}^{-1}\nu^{-1})$
as $[\nu^{0},\nu^{-1})$ and call it the inertial zone. Then, Theorem \ref{O} says that in the inertial zone the energy spectrum $E_n(u)$ behaves like $n^{-2}$. Likewise, we call the segment $[\nu^{-1}, +\infty)$ the dissipative zone, and \eqref{PA} results in the fact that in this zone, the energy spectrum $E_n(u)$ decreases faster than any negative power of $n$.
\begin{remark}
For the water turbulence the K41 theory proposed that $E_n(u)$ obeys the useful Kolmogorov-Obukhov law \cite{BB13}:
\begin{equation*}
E_{n}\sim|n|^{-\frac{5}{3}},
\end{equation*}
if $n$ in the inertial range. For fluid dynamics of turbulence in the Burgers equation, the physicist Jan Burgers in 1940 predicted that $E_{n}\sim|n|^{-2}$ for $|n|>C\nu^{-1}$, i.e., exactly the spectral power law above.
\end{remark}
\begin{remark}
 If random force $L(t,\omega,x)=\sum_{k=\pm1,\pm2,\cdot\cdot\cdot}\beta_{k}L_{k}(t,\omega)e_{k}(x)$ is such that $\beta_{k}\equiv\beta_{-k}$, i.e., $L(t,\omega,x)$ is homogeneous in $x$, then the velocity field $u(t,x)$ is stationary
in $t$ and homogeneous in $x$. What's more,
\begin{equation*}
\mathbb{E}e_{n}(u(t))\sim n^{-2}\quad\text{for all}~~t.
\end{equation*}
It is a perfect match for K41 of turbulence.
\end{remark}

\section{ Statistical quantities in the inviscid limit}\label{SI}
A remarkable fact is that, when $\nu\rightarrow0$, a solution $u^{\nu}$ of stochastic system \eqref{B} converges to an inviscid limit of turbulence :
\begin{equation*}
u^{\nu}(t,\cdot)\rightarrow u^{0}(t,\cdot)\quad\text{in}~~L^{p}(\mathbb{S}^{1}),~a.s.,
\end{equation*}
for each $p>0$. This result of the limiting dynamics is due to the Lax-Oleinik formula. The limit $u^0(t,x)$ is called an
inviscid solution, or an entropy solution of \eqref{B} with $\nu=0$. The limiting function $u^0(t,x)$ of the inviscid equation
is not even continuous. But still the structure function and energy spectrum are well defined
for $u^0(t,x)$, and they inherit all qualitative and quantitative properties proved previously for $u^{\nu}$ uniformly with small enough $\nu>0$ in last section.

Because formally there is no dissipation in the
inviscid Burgers equation, it does not have a dissipation range.  To make this rigorous, we define the non-empty and non-intersecting intervals $(0,\hat{C}_2]$ and $(\hat{C}_2,1]$ in $x$-presentation, which now correspond to the inertial range and the energy range, respectively. The
constant $\hat{C}_2$ is the same as Theorem \ref{Wy}. We denote  the structure function $S_{p}(l; u^{0})$ and energy spectrum $E_{n}(u^0)$ for $u^0(t,x)$ in the same way as the previously considered quantities $S_{p}(l; u^{\nu})$ and $E_{n}(u^\nu)$ for $u^{\nu}(t,x)$. On the basis of the power and utility for dominated convergence theorem, the following estimates remain valid in the inviscid limit.
\begin{theorem}
If $|l|\in(0,\hat{C}_2]$, then
\begin{description}
\item[1)] $E_{n}(u^0)\sim n^{-2}$ for all $n\in\mathbb{N}^{\ast}$,
\item[2)] and we gain the law
 \begin{eqnarray*}
S_{p}(l; u^{0})\sim\left\{\begin{array}{l}
C_p|l|,\quad\text{if}~~p\geq1;\\
C_p|l|^{p},~~~\text{if}~~p\in(0,1).
\end{array}
\right.
\end{eqnarray*}
\end{description}
\end{theorem}
This theorem describes stochastic turbulence in the inviscid limit. It should be noted that for $u^0$ the dissipation scale $l_d$ equals to $\infty$, and the inertial range in Fourier becomes the whole interval $[0,\infty)$. Now the energy law for $E_n(u^0)$ holds for all $n\in\mathbb{N}^{\ast}$ and the inertial range $(0,\hat{C}_2]$ in $x$.

\section{Conclusions and challenges}\label{five}
In stochastic Burgers equation \eqref{B} perturbed by L\'evy space-time white noise with the periodic boundary condition and small viscosity, we rigorously derived the moment estimates for Sobolev norms of solution, and proved statistical properties including structure function as well as energy spectrum. We focused on one-dimensional turbulence effected by cylindrical L\'evy process with bounded jumps, and illustrated the practical usage and applicability of the flow fluids for small but positive $\nu$, i.e., when $0<\nu\ll1$. Moreover, we obtained the qualitative and quantitative properties in the inviscid limit as the kinematic viscosity $\nu$ tends towards zero.

Let us comment here briefly on possible extensions of those results. When the noise involves large jumps, by using interlacing techniques, we expect scientific computation and further analysis for stochastic turbulence. But it does not escape us that if some jumps are too big, then non-Gaussian fluctuations cause sudden, intermittent and unpredictable dynamical behaviors of turbulent flows. In addition, it would be interesting to extend the present additive noise considerations to the case of multiplicative L\'evy noise. The jumps multiply the velocity, so we first must consider stochastic turbulence in the framework of Marcus type stochastic partial differential equations modeling jumps in the velocity gradient. It is complicated to identify probability  density  functions. Look at the problem dialectically, the Marcus properties give us a chain rule to take care of the large jumps. We plan to show those sophisticated contents, simulations and experiments in the future papers.
\medskip

{\it Acknowledgments.} \noindent  This work is initiated by an inspiring discussion with Sergei Kuksin. The authors are happy to thank Haitao Xu for fruitful discussions on stochastic differential equations driven by L\'evy motions, and non-equilibrium statistical
mechanics. The authors gratefully acknowledge support from the NSFC grant 12001213.

\bibliography{mybibfile}

\end{document}